\documentclass[11pt]{amsart}
\usepackage{amsmath}
\usepackage{amssymb}
\usepackage{amscd}

\def\NZQ{\mathbb}               
\def\NN{{\NZQ N}}

\def\ZZ{{\NZQ Z}}

\newtheorem{Theorem}{Theorem}[section]
\newtheorem{Lemma}[Theorem]{Lemma}
\newtheorem{Corollary}[Theorem]{Corollary}
\newtheorem{Proposition}[Theorem]{Proposition}

\newtheorem{Question}[Theorem]{Question}
\let\epsilon\varepsilon
\let\phi=\varphi
\let\kappa=\varkappa

\begin{document}

\title{The role of defect and splitting in finite generation of extensions of associated graded rings along a valuation}
\author{Steven Dale Cutkosky}
\thanks{partially supported by NSF}

\address{Steven Dale Cutkosky, Department of Mathematics,
University of Missouri, Columbia, MO 65211, USA}
\email{cutkoskys@missouri.edu}

\begin{abstract}
Suppose that $R$ is a 2 dimensional excellent local domain with quotient field $K$,  $K^*$ is a finite separable extension of $K$ and $S$ is a 2 dimensional local domain with quotient field
$K^*$ such that  $S$ dominates $R$. 
Suppose that $\nu^*$ is a valuation of $K^*$ such that 
 $\nu^*$ dominates $S$. Let $\nu$ be the restriction of $\nu^*$ to $K$.  The associated graded ring ${\rm gr}_{\nu}(R)$  was introduced by Bernard Teissier.  It plays an important role in local uniformization. We show in Theorem \ref{Theorem4} that the extension $(K,\nu)\rightarrow (K^*,\nu^*)$ of valued fields is without defect if and only if there exist regular local rings $R_1$ and $S_1$ such that
 $R_1$ is a local ring of a blow up of $R$, $S_1$ is a local ring of a blowup of $S$, $\nu^*$ dominates $S_1$, $S_1$ dominates $R_1$ and the associated graded ring  ${\rm gr}_{\nu^*}(S_1)$ is a finitely generated ${\rm gr}_{\nu}(R_1)$-algebra.
 
 We also investigate the role of splitting of the valuation $\nu$ in $K^*$ in finite generation of the extensions of associated graded rings along the valuation. We will say that $\nu$ does not split in $S$ if $\nu^*$ is the unique extension of $\nu$ to $K^*$ which  dominates $S$. We show  in Theorem \ref{ThmN2} that if $R$ and $S$ are regular local rings, $\nu^*$ has rational rank 1 and is not discrete and ${\rm gr}_{\nu^*}(S)$ is a finitely generated ${\rm gr}_{\nu}(R)$-algebra, then $\nu$ does not split in $S$. We give examples showing that such a strong statement is not true when $\nu$ does not satisfy these assumptions. As a consequence of Theorem \ref{ThmN2}, we deduce in Corollary \ref{CorN32} that if $\nu$ has rational rank 1 and is not discrete and if $R\rightarrow R'$ is a nontrivial sequence of quadratic transforms along $\nu$, then ${\rm gr}_{\nu}(R')$ is not a finitely generated ${\rm gr}_{\nu}(R)$-algebra. 
\end{abstract}

\maketitle

Suppose that $K$ is a field. Associated to a valuation $\nu$ of $K$  is a value group $\Phi_{\nu}$ and  a valuation ring $V_{\nu}$ with maximal ideal $m_{\nu}$. Let $R$ be a local domain with quotient field $K$. We say that $\nu$ dominates $R$ if $R\subset V_{\nu}$ and $m_{\nu}\cap R=m_R$ where $m_R$ is the maximal ideal of $R$. We have an associated semigroup $S^R(\nu)=\{\nu(f)\mid f\in R\}$, as well as the associated graded ring along the valuation
\begin{equation}\label{eqN31}
{\rm gr}_{\nu}(R)=\bigoplus_{\gamma\in \Phi_{\nu}}\mathcal P_{\gamma}(R)/\mathcal P^+_{\gamma}(R)=\bigoplus_{\gamma\in S^{R}(\nu)}\mathcal P_{\gamma}(R)/\mathcal P^+_{\gamma}(R)
\end{equation}
which is defined by Teissier in \cite{T1}. Here 
$$
\mathcal P_{\gamma}(R)=\{f\in R\mid \nu(f)\ge \gamma\}\mbox{ and }\mathcal P^+_{\gamma}(R)=\{f\in R\mid \nu(f)> \gamma\}.
$$ 
This ring plays an important role in local uniformization of singularities (\cite{T1} and \cite{T2}).
The ring ${\rm gr}_{\nu}(R)$ is a domain, but it is often not Noetherian, even when $R$ is.

Suppose that $K\rightarrow K^*$ is a finite extension of  fields and $\nu^*$ is a valuation which is an extension of $\nu$ to $K^*$. We have the classical indices
$$
e(\nu^*/\nu)=[\Phi_{\nu^*}:\Phi_{\nu}]\mbox{ and }f(\nu^*/\nu)=[V_{\nu^*}/m_{\nu^*}:V_{\nu}/m_{\nu}]
$$
as well as the defect  $\delta(\nu^*/\nu)$ of the extension. Ramification of valuations and the defect are discussed in Chapter VI of \cite{ZS2}, \cite{E} and Kuhlmann's papers \cite{Ku1} and \cite{Ku3}. A survey is given in Section 7.1 of \cite{CP}. By Ostrowski's lemma, if $\nu^*$ is the unique extension of $\nu$ to $K^*$, we have that
\begin{equation}\label{int3}
[K^*:K]=e(\nu^*/\nu)f(\nu^*/\nu)p^{\delta(\nu^*/\nu)}
\end{equation}
where $p$ is the characteristic of the residue field $V_{\nu}/m_{\nu}$. From this formula, the defect can be computed using Galois theory in an arbitrary finite extension. 
If  $V_{\nu}/m_{\nu}$ has characteristic 0, then $\delta(\nu^*/\nu)=0$ and  $p^{\delta(\nu^*/\nu)}=1$, so there is no defect. Further, if $\Phi_{\nu}=\ZZ$ and $K^*$ is separable over $K$ then there is no defect. 

If $K$ is an algebraic function field over a field $k$, then 
an algebraic local ring $R$ of $K$ is a local  domain which is essentially of finite type over $k$ and has $K$ as its field of fractions. 
In \cite{C}, it is shown that if $K\rightarrow K^*$ is a finite extension of algebraic function fields over a field $k$ of characteristic zero, $\nu^*$ is a valuation of $K^*$ (which is trivial on $k$) with restriction $\nu$ to $K$ and if $R\rightarrow S$ is an inclusion of algebraic regular local rings of $K$ and $K^*$ such that $\nu^*$ dominates $S$ and $S$ dominates $R$ then there exists a commutative diagram
\begin {equation}\label{int1}
\begin{array}{ccc}
R_1&\rightarrow &S_1\\
\uparrow&&\uparrow\\
R&\rightarrow&S
\end{array}
\end{equation}
where the vertical arrows are products of blowups of nonsingular subschemes along the valuation $\nu^*$ (monoidal transforms) and $R_1\rightarrow S_1$ is dominated by $\nu^*$ and is  a monomial mapping; that is, there exist regular parameters $x_1,\ldots,x_n$ in $R_1$, regular parameters  $y_1,\ldots,y_n$ in $S_1$, units $\delta_i\in S_1$, and a matrix $A=(a_{ij})$ of natural numbers  with $\mbox{Det}(A)\ne 0$ such that
 \begin{equation}\label{int2}
 x_i=\delta_i\prod_{j=1}^ny^{a_{ij}}\mbox{ for $1\le j\le n$}.
 \end{equation}
In \cite{CP}, it is shown that this theorem is true, giving a monomial form of the mapping (\ref{int2}) after appropriate blowing up  (\ref{int1}) along the valuation, if $K\rightarrow K^*$ is a separable extension of two dimension algebraic function fields over an algebraically closed field, which has no defect. This result is generalized to the situation of this paper, that is when $R$ is a  two dimensional excellent  local ring, in  \cite{C12}. However, it may  be that such  monomial forms do not exist, even after  blowing up, if the extension has defect, as is shown by examples in \cite{C2}.
 
 In the case when $k$ has characteristic zero and for separable defectless extensions of two dimensional algebraic function fields in positive characteristic, it is further shown in \cite{CP} that  the expressions (\ref{int1}) and (\ref{int2}) are stable under further simple sequences of blow ups along $\nu^*$ and the form of the matrix $A$ stably reflects invariants of the valuation. 

We always have an inclusion of graded domains 
${\rm gr}_{\nu}(R)\rightarrow {\rm gr}_{\nu^*}(S)$ and
the index of their quotient fields is 
\begin{equation}\label{int4}
[{\rm QF}({\rm gr}_{\nu^*}(S)):{\rm QF}({\rm gr}_{\nu}(R))]=e(\nu^*/\nu)f(\nu^*/\nu)
\end{equation}
 as shown in  Proposition 3.3 \cite{C11}. Comparing with Ostrowski's lemma (\ref{int3}), we see that  the defect has disappeared in  equation (\ref{int4}).

Even though  ${\rm QF}({\rm gr}_{\nu^*}(S))$ is finite over ${\rm QF}({\rm gr}_{\nu}(R))$, it is possible for ${\rm gr}_{\nu^*}(S)$ to not be a finitely generated ${\rm gr}_{\nu}(R)$-algebra. Examples showing this for extensions $R\rightarrow S$ of two dimensional algebraic local rings over arbitrary algebraically closed fields are given in Example 9.4 of \cite{CV1}.

It was shown by Ghezzi, H\`a and Kashcheyeva in \cite{GHK} for extensions of two dimensional algebraic function fields over an algebraically closed field $k$ of characteristic zero and later by Ghezzi and Kashcheyeva in \cite{GK} for defectless separable extensions of two dimensional algebraic functions fields over an algebraically closed field $k$ of positive characteristic that there exists a commutative diagram (\ref{int1}) such that
${\rm gr}_{\nu^*}(S_1)$ is a finitely generated ${\rm gr}_{\nu}(R_1)$-algebra. Further, this property is stable under further  suitable  sequences of blow ups. 

 In Theorem 1.6 \cite{C11}, it is shown that for algebraic regular local rings of arbitrary dimension, if the ground field $k$ is algebraically closed of characteristic zero, and the valuation has rank 1 and is zero dimensional ($V_{\nu}/m_{\nu}=k$) then we can also construct a commutative diagram (\ref{int1}) such that
${\rm gr}_{\nu^*}(S_1)$ is a finitely generated ${\rm gr}_{\nu}(R_1)$-algebra and  this property is stable under further  suitable sequences of blow ups.

An example is given in \cite{CP1} of an inclusion $R\rightarrow S$ in a separable defect extension of two dimensional algebraic function fields such that ${\rm gr}_{\nu^*}(S_1)$ is stably not  a finitely generated ${\rm gr}_{\nu}(R_1)$-algebra  in diagram (\ref{int1})
 under sequences of blow ups. This raises the question of whether the existence of a finitely generated extension of associated graded rings along the valuation implies that $K^*$ is a defectless extension of $K$.

We  find that we must impose the condition that $K^*$ is a separable extension of $K$ to obtain a positive answer to this question, as there are simple examples of inseparable defect extensions such that  ${\rm gr}_{\nu^*}(S)$  is  a finitely  generated ${\rm gr}_{\nu}(R)$-algebra, such as in the following example, which is Example 8.6 \cite{Ku1}.  Let $k$ be a field of characteristic $p>0$ and $k((x))$ be the field of formal power series over $k$, with the $x$-adic valuation $\nu_x$. Let $y\in k((x))$ be transcendental over $k(x)$ with $\nu_x(y)>0$. Let $\tilde y=y^{p}$, and $K=k(x,\tilde y)\subset  K^*=k(x,y)$. Let $\nu^*=\nu_x|K^*$ and $\nu=\nu_x|K$. Then we have equality of value groups
$\Phi_{\nu}=\Phi_{\nu^*}=\nu(x)\ZZ$ and equality of residue fields of valuation rings $V_{\nu}/m_{\nu}=V_{\nu^*}/m_{\nu^*}=k$, so $e(\nu^*/\nu)=1$ and $f(\nu^*/\nu)=1$. We have that $\nu^*$ is the unique extension of $\nu$ to $K^*$ since $K^*$ is purely inseparable over $K$. By Ostrowski's lemma (\ref{int3}),  
the extension $(K,\nu)\rightarrow (K^*,\nu^*)$ is a defect extension with defect $\delta(\nu^*/\nu)=1$. Let $R=k[x,\tilde y]_{(x,\tilde y)}\rightarrow S=k[x,y]_{(x,y)}$. Then we have equality
$$
{\rm gr}_{\nu}(R)=k[t]={\rm gr}_{\nu^*}(S)
$$
where $t$ is the class of $x$.

In this paper we show that the question does have a positive answer for separable extensions in the following theorem.

\begin{Theorem}\label{Theorem4} 
Suppose that $R$ is a 2 dimensional excellent local domain with quotient field $K$. Further suppose that $K^*$ is a finite separable extension of $K$ and $S$ is a 2 dimensional local domain with quotient field
$K^*$ such that  $S$ dominates $R$. 
Suppose that $\nu^*$ is a valuation of $K^*$ such that 
 $\nu^*$ dominates $S$. Let $\nu$ be the restriction of $\nu^*$ to $K$.  Then the extension $(K,\nu)\rightarrow (K^*,\nu^*)$ is without defect if and only if there exist regular local rings $R_1$ and $S_1$ such that
 $R_1$ is a local ring of a blow up of $R$, $S_1$ is a local ring of a blowup of $S$, $\nu^*$ dominates $S_1$, $S_1$ dominates $R_1$ and ${\rm gr}_{\nu^*}(S_1)$ is a finitely generated ${\rm gr}_{\nu}(R_1)$-algebra.

\end{Theorem}

We immediately obtain  the following corollary for two dimensional algebraic function fields.

\begin{Corollary}Suppose that $K\rightarrow K^*$ is a finite separable extension of two dimensional algebraic function fields over a field $k$ and  $\nu^*$ is a  valuation of $K^*$ with restriction $\nu$ to $K$. Then the extension $(K,\nu) \rightarrow (K^*,\nu^*)$ is without defect if and only if there exist  algebraic regular local rings $R$ of $K$ and $S$ of $K^*$ such that $\nu^*$ dominates $S$,  $S$ dominates $R$ and ${\rm gr}_{\nu^*}(S)$ is a finitely generated  ${\rm gr}_{\nu}(R)$-algebra.
\end{Corollary}

We  see from Theorem \ref{Theorem4} that the  defect, which is completely lost in the extension of quotient fields of the associated graded rings along the valuation (\ref{int4}), can be recovered from knowledge of all extensions of associated graded rings along the valuation of  regular local rings $R_1\rightarrow S_1$ within the field extension which dominate $R\rightarrow S$ and are dominated by the valuation.

The fact that there exists $R_1\rightarrow S_1$ as in the conclusions of the theorem if the assumptions of the theorem hold and the extension is without defect  is proven within 2-dimensional algebraic function fields over an algebraically closed field in \cite{GHK} and \cite{GK},  and in the generality of the assumptions of Theorem \ref{Theorem4} in Theorems 4.3 and 4.4 of \cite{C12}.
Further, if the assumptions of the theorem hold and the defect $\delta(\nu^*/\nu)\ne 0$, then the value group  $\Phi_{\nu^*}$ is not finitely generated  by Theorem 7.3 \cite{CP} in the case of algebraic function fields over an algebraically closed field. With  the full generality of the hypothesis  of Theorem \ref{Theorem4} , the defect is zero by  Corollary 18.7 \cite{E}
in the case of discrete, rank 1 valuations and the defect is zero by Theorem 3.7 \cite{C12} in the case of rational rank 2 valuations, so  by Abhyankar's inequality, Proposition 2 \cite{Ab1} or Appendix 2 \cite{ZS2}, if the defect  $\delta(\nu^*/\nu)\ne 0$, then the value group
$\Phi_{\nu^*}$ has rational rank 1 and is not discrete and $V_{\nu^*}/m_{\nu^*}$ is algebraic over $S/m_S$. Thus to prove Theorem \ref{Theorem4}, we have reduced to proving the following proposition, which we  establish in this paper.

\begin{Proposition}\label{Prop1} 
Suppose that $R$ is a 2 dimensional excellent local domain with quotient field $K$. Further suppose that $K^*$ is a finite separable extension of $K$ and $S$ is a 2 dimensional local domain with quotient field
$K^*$ such that  $S$ dominates $R$. 
Suppose that $\nu^*$ is a valuation of $K^*$ such that 
 $\nu^*$ dominates $S$. Let $\nu$ be the restriction of $\nu^*$ to $K$. 
 
 Suppose that  $\nu^*$ has rational rank 1 and $\nu^*$ is not discrete. Further suppose that 
  there exist regular local rings $R_1$ and $S_1$ such that
 $R_1$ is a local ring of a blow up of $R$, $S_1$ is a local ring of a blowup of $S$, $\nu^*$ dominates $S_1$, $S_1$ dominates $R_1$ and ${\rm gr}_{\nu^*}(S_1)$ is a finitely generated ${\rm gr}_{\nu}(R_1)$-algebra.
Then the defect $\delta(\nu^*/\nu)=0$.
\end{Proposition}

Another factor in the question of finite generation of extensions of associated graded rings along a valuation is the splitting of $\nu$ in $K^*$. We will say that $\nu$ does not split in $S$ if $\nu^*$ is the unique extension of $\nu$ to $K^*$ such that $\nu^*$ dominates $S$.  After a little blowing up, we can always obtain non splitting, as the following lemma shows.

\begin{Lemma}\label{LemmaN66} Given an extension $R\rightarrow S$ as in the hypotheses  of Theorem \ref{Theorem4},  there exists a normal local ring $R'$ which is a local ring of a blow up of $R$ such that $\nu$ dominates $R'$ and if 
$$
\begin{array}{lll}
R_1&\rightarrow&S_1\\
\uparrow&&\uparrow\\
R&\rightarrow&S
\end{array}
$$
is a commutative diagram of normal local rings, where $R_1$ is a local ring of a blow up of $R$ and $S_1$ is a local ring of a blow up of $S$, $\nu^*$ dominate $S_1$ and $R_1$ dominates $R'$,  then $\nu$ does not split in $S_1$.
\end{Lemma}
Lemma \ref{LemmaN66} will be proven in Section \ref{SecLocDeg}.

We have the following theorem.

\begin{Theorem}\label{ThmN2}    Suppose that $R$ is a 2 dimensional excellent regular local ring with quotient field $K$. Further suppose that $K^*$ is a finite separable extension of $K$ and $S$ is a 2 dimensional regular local ring with quotient field
$K^*$ such that  $S$ dominates $R$. 
Suppose that $\nu^*$ is a valuation of $K^*$ such that 
 $\nu^*$ dominates $S$. Let $\nu$ be the restriction of $\nu^*$ to $K$.        Further suppose that $\nu^*$ has rational rank 1 and $\nu^*$ is not discrete. Suppose that ${\rm gr}_{\nu^*}(S)$ is a finitely generated ${\rm gr}_{\nu}(R)$-algebra. Then $S$ is a localization of the integral closure of $R$ in $K^*$, the defect $\delta(\nu^*/\nu)=0$ and $\nu^*$ does not split in $S$.  
\end{Theorem}

We give  examples showing that the condition rational rank 1 and discrete on $\nu^*$ in Theorem \ref{ThmN2} are necessary.

As an immediate consequence of  Theorem \ref{ThmN2}, we obtain the following corollary.

\begin{Corollary}\label{CorN32}   Suppose that $R$ is a 2 dimensional excellent regular local ring with quotient field $K$. Suppose that $\nu$ is a valuation of $K$ such that 
 $\nu$ dominates $R$.        Further suppose that $\nu$ has rational rank 1 and $\nu$ is not discrete. Suppose that $R\rightarrow R'$ is a nontrivial sequence of quadratic transforms along $\nu$. Then
 $\mbox{gr}_{\nu}(R')$ is not a finitely generated $\mbox{gr}_{\nu}(R)$-algebra.
 \end{Corollary}
 
 In \cite{Vaq}, Michel Vaqui\'e extends MacLane's theory of key polynomials \cite{M} to show that if $(K,\nu)\rightarrow (K^*,\nu^*)$ is a finite extension of valued fields with $\delta(\nu^*/\nu)=0$ and $\nu^*$ is the unique extension of $\nu$ to $K^*$, then $\nu^*$ can be constructed from $\nu$ by a finite sequence of augmented valuations. This suggests that a converse of Theorem \ref{ThmN2} may be true.
 
 We thank Bernard Teissier for discussions on the topics of this paper.

\section{Local degree and defect}\label{SecLocDeg}

We will use the following criterion to measure defect, which is Proposition 3.4 \cite{C12}. This result is implicit in \cite{CP} with the assumptions of Proposition \ref{Prop1}.

\begin{Proposition}\label{Prop15}
 Suppose that $R$ is a 2 dimensional excellent local domain with quotient field $K$. Further suppose that $K^*$ is a finite separable extension of $K$ and $S$ is a 2 dimensional local domain with quotient field
$K^*$ such that  $S$ dominates $R$. 
Suppose that $\nu^*$ is a valuation of $K^*$ such that 
 $\nu^*$ dominates $S$, the residue field $V_{\nu^*}/m_{\nu^*}$ of $V_{\nu^*}$ is algebraic over $S/m_S$ and 
the value group $\Phi_{\nu^*}$ of $\nu^*$ has rational rank 1.
Let $\nu$ be the restriction of $\nu^*$ to $K$. There
 exists a local ring $R'$  of $K$ which is essentially of finite type over $R$, is dominated by 
$\nu$ and dominates $R$ such that if we have a commutative diagram
\begin{equation}\label{eq31}
\begin{array}{lll}
V_{\nu}&\rightarrow&V_{\nu^*}\\
\uparrow&&\uparrow\\
R_1&\rightarrow&S_1\\
\uparrow&&\\
R'&&\uparrow\\
\uparrow&&\\
R&\rightarrow&S
\end{array}
\end{equation}
where 
$R_1$ is a regular local ring of $K$ which is essentially of finite type over $R$ and dominates $R$, $S_1$ is a regular local ring of $K^*$ which is essentially of finite type over $S$ and dominates $S$, 
$R_1$ has a regular system of parameters $u,v$ and $S_1$ has a regular system of parameters $x,y$ such that there is an expression
$$
u=\gamma x^a, v=x^bf
$$
where $a>0$, $b\ge 0$, $\gamma$ is a unit in $S$, $x\not\,\mid f$ in $S_1$ and $f$ is not a unit in $S_1$, then
\begin{equation}\label{eq37}
ad[S_1/m_{S_1}:R_1/m_{R_1}]=e(\nu^*/\nu)f(\nu^*/\nu)p^{\delta(\nu^*/\nu)}
\end{equation}
where $d=\overline\nu(f\mbox{ mod }x)$ with $\overline \nu$ being the natural valuation of the DVR $S/xS$.
\end{Proposition}
\vskip .2truein
We now prove Lemma \ref{LemmaN66} from the introduction. Let $\nu_1=\nu^*,\nu_2,\ldots,\nu_r$ be the extensions of $\nu$ to $K^*$. Let $T$ be the integral closure of $V_{\nu}$ in $K^*$. Then $T=V_{\nu_1}\cap\cdots\cap V_{\nu_r}$ is the integral closure of $V_{\nu^*}$ in $K^*$ (by Propositions 2.36 and 2.38 \cite{RTM}). Let $m_i=m_{\nu_i}\cap T$ be the maximal ideals of $T$. By the Chinese remainder theorem, there exists $u\in T$ such that $u\in m_1$ and $u\not\in m_i$ for $2\le i\le r$. Let 
$$
u^n+a_1u^{n-1}+\cdots+a_n=0
$$
be an equation of integral dependence of $u$ over $V_{\nu}$.  Let $A$ be the integral closure of $R[a_1,\ldots,a_n]$ in $K$ and let $R'=A_{A\cap m_{\nu}}$. Let $T'$ be the integral closure of $R'$ in $K^*$. We have that $u\in T'\cap m_i$ if and only if $i=1$. Let $S'=T'_{T'\cap m_1}$. Then $\nu$ does not split in $S'$ and $R'$ has the property of the conclusions of the lemma.

\section{Generating Sequences}\label{SecGen}

Given an additive  group $G$ with $\lambda_0,\ldots, \lambda_r\in G$, $G(\lambda_0,\ldots,\lambda_r)$ will denote the subgroup generated by $\lambda_0,\ldots,\lambda_r$. The semigroup generated by $\lambda_0,\ldots, \lambda_r$ will be denoted by $S(\lambda_0,\ldots,\lambda_r)$. 

In this section, we will suppose that $R$ is a regular local ring of dimension two, with maximal ideal $m_R$ and residue field $R/ m_R$.
For $f\in R$, let $\overline f$ or $[f]$ denote the residue of $f$ in $R/ m_R$. 

The following theorem is Theorem 4.2 of \cite{CV1}, as interpreted by Remark 4.3 \cite{CV1}.

\begin{Theorem}\label{Theorem1*} Suppose that  $\nu$ is a  valuation 
of the quotient field of $R$ dominating $R$. Let $L=V_{\nu}/m_{\nu}$ be the residue field of the valuation ring $V_{\nu}$ of $\nu$. For $f \in V_{\nu}$, let $[f]$ denote the class of $f$ in $L$. Suppose that $x,y$ are regular parameters in $R$.
Then  there exist $\Omega\in\ZZ_+\cup\{\infty\}$ and 
$P_i(\nu,R)\in m_R$ for $i\in\ZZ_+$ with $i<\min\{\Omega+1,\infty\}$
 such that $P_0(\nu,R)=x$, $P_1(\nu,R)=y$ and for $1\le i<\Omega$, there is an expression
 \begin{equation}\label{eq11*} 
 P_{i+1}(\nu,R) = P_i(\nu,R)^{n_i(\nu,R)}+\sum_{k=1}^{\lambda_i} c_kP_0(\nu,R)^{\sigma_{i,0}(k)}P_1(\nu,R)^{\sigma_{i,1}(k)}\cdots P_{i}(\nu,R)^{\sigma_{i,i}(k)}
 \end{equation}
 with $n_i(\nu,R)\ge 1$, $\lambda_i\ge 1$, 
 \begin{equation}\label{eq12*}
 0\ne c_k\mbox{ units in }R
 \end{equation}
  for $1\le k\le \lambda_i$,
 $\sigma_{i,s}(k)\in\NN$ for all $s,k$,  $0\le \sigma_{i,s}(k)<n_s(\nu,R)$ for $s\ge 1$.
 Further,
 $$
 n_i(\nu,R)\nu(P_i(\nu,R))=\nu(P_0(\nu,R)^{\sigma_{i,0}(k)}P_1(\nu,R)^{\sigma_{i,1}(k)}\cdots P_{i}(\nu,R)^{\sigma_{i,i}(k)})
 $$
 for all $k$.
 
 For all $i\in\ZZ_+$ with $i<\Omega$, the following are true:
 \begin{enumerate}
 \item[1)] $\nu(P_{i+1}(\nu,R))>n_i(\nu,R)\nu(P_i(\nu,R))$.
 \item[2)] Suppose that $r\in\NN$, $m\in \ZZ_+$, $j_k(l)\in\NN$  for $1\le l\le m$ and $0\le j_k(l)<n_k(\nu,R)$ for $1\le k\le r$ are such that  $(j_0(l),j_1(l),\ldots,j_r(l))$ are distinct for $1\le l\le m$, and 
 $$
 \nu(P_0(\nu,R)^{j_0(l)}P_1(\nu,R)^{j_1(l)}\cdots P_r(\nu,R)^{j_r(l)})=\nu(P_0(\nu,R)^{j_0(1)}\cdots P_r(\nu,R)^{j_r(1)})
 $$
 for $1\le l\le m$.
 Then
 $$
 1,\left[\frac{P_0(\nu,R)^{j_0(2)}P_1(\nu,R)^{j_1(2)}\cdots P_r(\nu,R)^{j_r(2)}}{P_0(\nu,R)^{j_0(1)}P_1(\nu,R)^{j_1(1)}\cdots P_r(\nu,R)^{j_r(1)}}\right],
 \ldots,
 \left[\frac{P_0(\nu,R)^{j_0(m)}P_1(\nu,R)^{j_1(m)}\cdots P_r(\nu,R)^{j_r(m)}}{P_0(\nu,R)^{j_0(1)}P_1(\nu,R)^{j_1(1)}\cdots P_r(\nu,R)^{j_r(1)}}\right]
 $$
 are linearly independent over $R/ m_R$.
 \item[3)] Let 
 $$
 \overline n_i(\nu,R)=[G(\nu(P_0(\nu,R)),\ldots, \nu(P_(\nu,R)i)):G(\nu(P_0(\nu,R)),\ldots, \nu(P_{i-1}(\nu,R)))].
 $$
  Then $\overline n_i(\nu,R)$ divides $\sigma_{i,i}(k)$ for all $k$ in (\ref{eq11*}). In particular, $n_i(\nu,R)=\overline n_i(\nu,R)d_i(\nu,R)$ with $d_i(\nu,R)\in \ZZ_+$ 
 \item[4)] There exists $U_i(\nu,R)=P_0(\nu,R)^{w_0(i)}P_1(\nu,R)^{w_1(i)}\cdots P_{i-1}(\nu,R)^{w_{i-1}(i)}$ for $i\ge 1$ with $w_0(i),\ldots, w_{i-1}(i)\in\NN$ 
 and $0\le w_j(i)<n_j(\nu,R)$ for $1\le j\le i-1$ such that
 $\nu(P_i(\nu,R)^{\overline n_i})=\nu(U_i(\nu,R))$ and setting
 $$
 \alpha_i(\nu,R)=\left[\frac{P_i(\nu,R)^{\overline n_i(\nu,R)}}{U_i(\nu,R)}\right]
 $$
 then 
 $$
 \begin{array}{lll}
 b_{i,t}&=&\left[\sum_{\sigma_{i,i}(k)=t\overline n_i(\nu,R)}c_k\frac{P_0(\nu,R)^{\sigma_{i,0}(k)}P_1(\nu,R)^{\sigma_{i,1}(k)}\cdots P_{i-1}(\nu,R)^{\sigma_{i,i-1}(k)}} {U_i(\nu,R)^{(d_i(\nu,R)-t)}}\right]\\
 &&\in R/m_R(\alpha_1(\nu,R),\ldots,\alpha_{i-1}(\nu,R))
 \end{array}
 $$
 for $0\le t\le d_i(\nu,R)-1$ and
 $$
 f_i(u)=u^{d_i(\nu,R)}+b_{i,d_i(\nu,R)-1}u^{d_i(\nu,R)-1}+\cdots+b_{i,0}
 $$
 is the minimal polynomial of $\alpha_i(\nu,R)$ over $R/m_R(\alpha_1(\nu,R),\ldots,\alpha_{i-1}(\nu,R))$.
 
\end{enumerate}

The algorithm terminates with $\Omega<\infty$ if and only if either
\begin{equation}\label{eqL15}
\overline n_{\Omega}(\nu,R)=[G(\nu(P_0(\nu,R)),\ldots, \nu(P_{\Omega}(\nu,R))):G(\nu(P_0(\nu,R)),\ldots, \nu(P_{\Omega-1}(\nu,R)))]=\infty
\end{equation}
or 
\begin{equation}\label{eqL10}
\begin{array}{l}
\mbox{$\overline n_{\Omega}(\nu,R)<\infty$ (so that $\alpha_{\Omega}(\nu,R)$ is defined as in 4)) and}\\
\mbox{$d_{\Omega}(\nu,R)=[R/m_R(\alpha_1(\nu,R),\ldots,\alpha_{\Omega}(\nu,R)):R/m_R(\alpha_1(\nu,R),\ldots,\alpha_{\Omega-1}(\nu,R))]=\infty$.}
\end{array}
\end{equation}
If $\overline n_{\Omega}(\nu,R)=\infty$, set $\alpha_{\Omega}(\nu,R)=1$.

 \end{Theorem}

 Let notation be as in Theorem \ref{Theorem1*}.
 
 The following formula is  formula $B(i)$ on page 10 of \cite{CV1}.

\vskip .2truein
\begin{equation}\label{eqZ51}
\begin{array}{l}
\mbox{ Suppose that $M$ is a Laurent monomial in $P_0(\nu,R),P_1(\nu,R),\ldots, P_i(\nu,R)$}\\
\mbox{ and $\nu(M)=0$. Then there exist $s_i\in\ZZ$ such that }\\
\,\,\,\,\,\,\,\,\,\, M=\prod_{j=1}^i\left[\frac{P_j(\nu,R)^{\overline n_j}}{U_j(\nu,R)}\right]^{s_j},\\
\mbox{ so that}\\
\,\,\,\,\,\,\,\,\,\, [M]\in R/m_R[\alpha_1(\nu,R),\ldots,\alpha_i(\nu,R)].
\end{array}
\end{equation}
\vskip .2truein

Define $\beta_i(\nu,R)=\nu(P_i(\nu,R))$ for $0\le i$.
 
 Since $\nu$ is a valuation of the quotient field of $R$, we have that 
 \begin{equation}\label{eq50}
 \Phi_{\nu}=\cup_{i=1}^{\infty}G(\beta_0(\nu,R),\beta_1,\ldots,\beta_i(\nu,R))
 \end{equation}
 and
 \begin{equation}\label{eq51}
 V_{\nu}/m_{\nu}=\cup_{i=1}^{\infty} R/m_R[\alpha_1(\nu,R),\ldots,\alpha_i(\nu,R)]
 \end{equation}

The following is Theorem 4.10 \cite{CV1}.

\begin{Theorem}\label{TheoremG2} 
Suppose that $\nu$ is a valuation dominating $R$.
Let 
$$
P_0(\nu,R)=x, P_1(\nu,R)=y, P_2(\nu,R),\ldots
$$
be the sequence of elements of $R$ constructed by  Theorem \ref{Theorem1*}.  Suppose that $f\in R$  and there exists $n\in\ZZ_+$ such that $\nu(f)<n\nu(m_R)$. 
Then there exists an expansion 
$$
f=\sum_{I}a_IP_0(\nu,R)^{i_0}P_1(\nu,R)^{i_1}\cdots P_r(\nu,R)^{i_r}+\sum_J\phi_JP_0(\nu,R)^{j_0}\cdots P_r(\nu,R)^{j_r}+h
$$
where $r\in\NN$, $a_{I}$ are units in $R$, $I,J\in \NN^{r+1}$, $\nu(P_0(\nu,R)^{i_0}P_1(\nu,R)^{i_1}\cdots P_r(\nu,R)^{i_r})=\nu(f)$ for all $I$ in the first sum,  $0\le i_k<n_k(\nu,R)$ for $1\le k\le r$, $\nu(P_0(\nu,R)^{j_0}\cdots P_r(\nu,R)^{j_r})>\nu(f)$ for all terms in the second sum, $\phi_J\in R$ and $h\in  m_R^n$.
The terms in the first sum are uniquely determined, up to the choice of units $a_i$, whose residues in $R/m_R$ are uniquely determined.
\end{Theorem}

 Let $\sigma_0(\nu,R)=0$ and inductively define 
 \begin{equation}\label{eq3}
 \sigma_{i+1}(\nu,R)=\min\{j>\sigma_i(\nu,R)\mid n_j(\nu,R)>1\}.
 \end{equation}
 
 In Theorem \ref{TheoremG2}, we see that all of the monomials in the expansion of $f$ are in terms of the $P_{\sigma_i}$.
 
 We have that
 $$
 S(\beta_0(\nu,R),\beta_1(\nu,R),\ldots,\beta_{\sigma_j(\nu,R)})=S(\beta_{\sigma_0}(\nu,R),\beta_{\sigma_1(\nu,R)},\ldots,\beta_{\sigma_j(\nu,R)})
 $$
for all $j\ge 0$ and
$$
\begin{array}{l}
R/m_R[\alpha_1(\nu,R),\alpha_2(\nu,R),\ldots,\alpha_{\sigma_j(\nu,R)}(\nu,R)]\\
=R/m_R[\alpha_{\sigma_1(\nu,R)}(\nu,R),\alpha_{\sigma_2(\nu,R)}(\nu,R),\ldots,\alpha_{\sigma_j(\nu,R)}(\nu,R)]
\end{array}
$$
for all $j\ge 1$.

Suppose that $R$ is a regular local ring of dimension two which is dominated by a valuation $\nu$. The quadratic transform $T_1$ of $R$ along $\nu$ is defined as follows. Let $u,v$ be a system of regular parameters in $R$,
 Then $R[\frac{v}{u}]\subset V_{\nu}$ if $\nu(u)\le \nu(v)$ and $R[\frac{u}{v}]\subset V_{\nu}$ if $\nu(u)\ge \nu(v)$. Let 
$$
T_1=R\left[\frac{v}{u}\right]_{R[\frac{v}{u}]\cap  m_{\nu}}\mbox{ or }T_1=R\left[\frac{u}{v}\right]_{R[\frac{u}{v}]\cap  m_{\nu}},
$$
depending on if $\nu(u)\le\nu(v)$ or $\nu(u)>\nu(v)$.
$T_1$ is a two dimensional regular local ring which is dominated by $\nu$. 
Let
\begin{equation}\label{eqX3}
R\rightarrow T_1\rightarrow T_2\rightarrow \cdots
\end{equation}
be the infinite sequence of quadratic transforms along $\nu$, so that $V_{\nu}=\cup_{i\ge 1} T_i$ (Lemma 4.5 \cite{RTM}) and $L=V_{\nu}/m_{\nu}=\cup_{i\ge 1} T_i/m_{T_i}$.

For $f\in R$ and $R\rightarrow R^*$ a sequence of quadratic transforms along $\nu$, we define a strict transform of $f$ in $R^*$ to be $f_1$ if $f_1\in R^*$  is a local equation of the strict transform  in $R^*$ of the subscheme $f=0$ of $R$. In this way, a strict transform is only defined up to multiplication by a unit in $R^*$. This ambiguity will not be a difficulty in our proof. We will denote a strict transform of $f$ in $R^*$ by $\mbox{st}_{R^*}(f)$.

We use the notation of Theorem \ref{Theorem1*} and its proof for $R$ and the $\{P_i(\nu,R)\}$. Recall that $U_1=U^{w_0(1)}$. Let $w=w_0(1)$. Since $\overline n_1(\nu,R)$ and $w$ are relatively prime, there exist $a,b\in\NN$ such that 
$$
\epsilon:=\overline n_1(\nu,R)b-wa=\pm 1.
$$
 Define
elements of the quotient field of $R$ by
\begin{equation}\label{eqX50}
x_1=(x^by^{-a})^{\epsilon}, y_1=(x^{-w}y^{\overline n_1(\nu,R)})^{\epsilon}.
\end{equation}
 We have that
\begin{equation}\label{eqZ1}
x=x_1^{\overline n_1(\nu,R)}y_1^a, y=x_1^wy_1^b.
\end{equation}
Since $\overline n_1(\nu,R)\nu(y)=w\nu(x)$, it follows that
$$
\overline n_1(\nu,R)\nu(x_1)=\nu(x)>0\mbox{ and }\nu(y_1)=0.
$$
We further have that
\begin{equation}\label{eqZ3}
\alpha_1(\nu,R)=[y_1]^{\epsilon}\in V_{\nu}/m_{\nu}.
\end{equation}
Let $A=R[x_1,y_1]\subset V_{\nu}$ and $ m_A= m_{\nu}\cap A$. 

Let $R_1=A_{ m_A}$. We have that $R_1$ is a regular local ring and the divisor of $xy$ in $R_1$
 has only one component ($x_1=0$). In particular, $R\rightarrow R_1$ is ``free'' (Definition 7.5 \cite{CP}).
$R\rightarrow R_1$ factors (uniquely) as a product
of quadratic transforms and the divisor of $xy$ in $R_1$  has two distinct irreducible factors in all intermediate rings. 

The following is Theorem 7.1 \cite{CV1}.

\begin{Theorem}\label{birat}
Let $R$ be a two dimensional regular local ring with regular parameters $x,y$. Suppose that $R$ is dominated by a valuation $\nu$. Let $P_0(\nu,R)=x$, $P_1(\nu,R)=y$ and $\{P_i(\nu,R)\}$ be the sequence of elements of $R$ constructed in Theorem \ref{Theorem1*}. Suppose that $\Omega\ge 2$. Then there exists some smallest value $i$ in the sequence (\ref{eqX3}) such that
the divisor of $xy$ in $\mbox{Spec}(T_i)$ has only one component. Let $R_1=T_i$.
Then $R_1/m_{R_1}\cong R/m_R(\alpha_1(\nu,R))$, and there exists $x_1\in R_1$ and $w\in\ZZ_+$ such that
$x_1=0$ is a local equation of the exceptional divisor of $\mbox{Spec}(R_1)\rightarrow \mbox{Spec}(R)$, and $Q_0=x_1$, $Q_1=\frac{P_2}{x_1^{wn_1}}$ are regular parameters in $R_1$. We have that
$$
P_i(\nu,R_1)=\frac{P_{i+1}(\nu,R)}{P_0(\nu,R_1)^{w n_1(\nu,R)\cdots n_i(\nu,R)}}
$$
for $1\le i< \max\{\Omega,\infty\}$
satisfy the conclusions  of Theorem \ref{Theorem1*}  for the ring $R_1$.
\end{Theorem}

We have that
$$
G(\beta_0(\nu,R_1),\ldots,\beta_i(\nu, R_1))=G(\beta_0(\nu,R),\ldots,\beta_{i+1}(\nu,R))
$$
for $i\ge 1$ so that 
$$
\overline n_i(\nu,R_1)=\overline n_{i+1}(\nu,R)\mbox{ for }i\ge 1
$$ 
and
$$
R_1/m_{R_1}[\alpha_1(\nu,R_1),\ldots,\alpha_i(\nu,R_1)]=R/m_R[\alpha_1(\nu,R),\ldots,\alpha_{i+1}(\nu,R)]\mbox{ for }i\ge 1
$$
so that 
$$
d_i(\nu,R_1)=d_{i+1}(\nu,R)\mbox{ and }n_{i}(\nu,R_1)=n_{i+1}(\nu,R)\mbox{ for }i\ge 1.
$$

Let $\sigma_0(\nu,R_1)=0$ and inductively define 
    $$
    \sigma_{i+1}(\nu,R_1)=\min\{j>\sigma_i(1)\mid n_j(\nu,R_1)>1\}.
    $$

    We then have that 
    $\sigma_0(\nu,R_1)=0$ and for $i\ge 1$,  $\sigma_i(\nu,R_1)=\sigma_{i+1}(\nu,R)-1$ if $n_1(\nu,R)>1$ and $\sigma_i(\nu,R_1)=\sigma_i(\nu,R)-1$ if $n_1(\nu,R)=1$, and for all $j\ge 0$,
    $$
    S(\beta_0(\nu,R_1),\beta_1(\nu,R_1),\ldots,\beta_{\sigma_{j+1}(\nu,R_1)}(\nu,R_1))=S(\beta_{\sigma_0(1)}(\nu,R_1),\beta_{\sigma_1(\nu,R_1)},\ldots,\beta_{\sigma_j(\nu,R_1)}(\nu,R_1))
    $$

 Iterating this construction, we produce a sequence of sequences of quadratic transforms along $\nu$,
    $$
    R\rightarrow R_1\rightarrow \cdots\rightarrow R_{\sigma_1(\nu,R)}.
    $$
    
 Now $x, \overline y=P_{\sigma_1(\nu,R)}$ are regular parameters in $R$. By (\ref{eqX50}) (with $y$ replaced with $\overline y$) we have that $R_{\sigma_1(\nu,R)}$ has regular parameters
 \begin{equation}\label{eq8}
 x_1=(x^b\overline y^{-a})^{\epsilon},\,\,\,
 y_1=(x^{-\omega}\overline y^{\overline n_{\sigma_1(\nu,R)}(\nu,R)})^{\epsilon}
 \end{equation}
 where $\omega,a,b\in \NN$ satisfy
 $\epsilon=\overline n_{\sigma_1(\nu,R)}(\nu,R)b-\omega a=\pm 1$.

 Further, $R_{\sigma_1(\nu,R_1)}$ has regular parameters $x_{\sigma_1(\nu,R)}, y_{\sigma_1(\nu,R)}$ where $x=\delta x_{\sigma_1(\nu,R_1)}^{\overline n_{\sigma_1(\nu,R)}(\nu,R_1)}$ and $y_{\sigma_1(\nu,R_1)}={\rm st}_{R_{\sigma_1}(\nu,R_1)}P_{\sigma_1(\nu,R)}(\nu,R)$ with $\delta\in R_{\sigma_1(\nu,R)}$ a unit.

For the remainder of this section, we will suppose that $R$ is a two dimensional regular local ring and $\nu$ is a non discrete rational rank 1  valuation of the quotient field of $R$ with valuation ring $V_{\nu}$, so that $V_{\nu}/m_{\nu}$ is algebraic over $R/m_R$. Suppose that $f\in R$ and $\nu(f)=\gamma$. We will denote the class of $f$ in $\mathcal P_{\gamma}(R)/\mathcal P^+_{\gamma}(R)\subset {\rm gr}_{\nu}(R)$ by ${\rm in}_{\nu}(f)$.
By Theorem \ref{TheoremG2}, we have that ${\rm gr}_{\nu}(R)$ is generated by the initial forms of the $P_i(\nu,R)$ as an $R/m_R$-algebra. That is, 
$$
\begin{array}{l}
{\rm gr}_{\nu}(R)=R/m_R[{\rm in}_{\nu}(P_0(\nu,R)),{\rm in}_{\nu}(P_1(\nu,R)),\ldots]\\
=R/m_R[{\rm in}_{\nu}(P_{\sigma_0(\nu,R)}(\nu,R)),{\rm in}_{\nu}(P_{\sigma_1(\nu,R)}(\nu,R)),\ldots].
\end{array}
$$
Thus the semigroup
$S^R(\nu)=\{\nu(f)\mid f\in R\}$ is equal to 
$$
S^R(\nu)=S(\beta_0(\nu,R),\beta_1(\nu,R),\ldots)=S(\beta_{\sigma_0(\nu,R)}(\nu,R),\beta_{\sigma_1(\nu,R)}(\nu,R),\ldots)
$$
and the value group
$$
\Phi_{\nu}=G(\beta_0(\nu,R),\beta_1(\nu,R)\ldots)
$$
and the residue field of the valuation ring
$$
V_{\nu}/m_{\nu}=R/m_R[\alpha_1(\nu,R),\alpha_2(\nu,R),\ldots]=R/m_R[\alpha_{\sigma_1}(\nu,R),\alpha_{\sigma_2}(\nu,R),\ldots]
$$

 By 1) of Theorem \ref{Theorem1*}, every element $\beta\in S^{R}(\nu)$ has a unique expression
 $$
 \beta=\sum_{i=0}^ra_i\beta_i(\nu,R)
 $$
 for some $r$ with $a_i\in \NN$ for all $i$ and $0\le a_i<n_i(\nu,R)$ for $1\le i$. In particular, if $a_i\ne 0$ in the expansion then $\beta_i(\nu,R)=\beta_{\sigma_j(\nu,R)}(\nu,R)$ for some $j$.

\begin{Lemma}\label{Lemma1} Let
$$
\sigma_i=\sigma_i(\nu,R), \beta_i=\beta_i(\nu,R), P_{i}=P_i(\nu,R), n_i=n_i(\nu,R), \overline n_i=\overline n_i(\nu,R),
$$
$$
\sigma_i(1)=\sigma_i(\nu,R_{\sigma_1}), \beta_i=\beta_i(\nu,R_{\sigma_1}), P_i(1)=P_i(\nu,R_{\sigma_1}), n_i(1)=n_i(\nu,R_{\sigma_1}), \overline n_i(1)=\overline n_i(\nu,R_{\sigma_1}).
$$

Suppose $i\in \NN$, $r\in \NN$  and $a_j\in \NN$ for $j=0,\ldots, r$ with $0\le a_j<n_{\sigma_j}$ for $j\ge 1$ are such that 
$$
\nu(P_{\sigma_0}^{a_0}\cdots P_{\sigma_r}^{a_r})> \nu(P_{\sigma_i})
$$
 or $r<i$ and 
 $$
 \nu(P_{\sigma_0}^{a_0}\cdots P_{\sigma_r}^{a_r})= \nu(P_{\sigma_i}).
 $$
 By (\ref{eqZ1}) and Theorem \ref{birat}, we have expressions in 
 $$
 R_{\sigma_1}=R[x_1,y_1]_{m_{\nu}\cap R[x_1,y_1]}
 $$
 where $x_1,y_1$ are defined by (\ref{eq8})
$$
P_{\sigma_0}^{a_0}\cdots P_{\sigma_r}^{a_r}=y_1^{aa_0+ba_1}P_{\sigma_1(1)}(1)^{a_2}\cdots P_{
\sigma_{r-1}(1)}(1)^{a_r}P_{\sigma_0(1)}(1)^t
$$
where $t=\overline n_{\sigma_1}a_0+\omega a_1+\omega n_{\sigma_1}a_2+\cdots+\omega n_{\sigma_1}\cdots n_{\sigma_{r-1}}a_r$ and 
$$
P_{\sigma_i}=
\left\{
\begin{array}{ll}
y_1^aP_{\sigma_0(1)}(1)^{\overline n_{\sigma_1}}&\mbox{ if }i=0\\
y_1^bP_{\sigma_0(1)}(1)^{\omega}&\mbox{ if }i=1\\
P_{\sigma_{i-1}(1)}(1)P_{\sigma_0(1)}(1)^{\omega n_{\sigma_1}\cdots n_{\sigma_{i-1}}}&\mbox{ if }i\ge 2.
\end{array}\right.
$$
Let 
$$
\lambda=\left\{\begin{array}{ll}
\overline n_{\sigma_1}&\mbox{ if }i=0\\
\omega&\mbox{ if }i=1\\
\omega n_{\sigma_1}\cdots n_{\sigma_{i-1}}&\mbox{ if }i\ge 2.
\end{array}\right.
$$
Then 
$$
t>\lambda,
$$
except  in the case where $i=1$, $P_{\sigma_0}^{a_0}\cdots P_{\sigma_r}^{a_r}=P_{\sigma_0}$, and $\overline n_{\sigma_1}=\omega=1$.
In this  case $\lambda=t$. 
\end{Lemma}

\begin{proof} First suppose that $i\ge 2$ and $r\ge i$.  Then 
$$
t-\lambda=(\overline n_{\sigma_1}a_0+\omega a_1+\omega n_{\sigma_1}a_2+\cdots +\omega n_{\sigma_1}\cdots n_{\sigma_{r-1}}a_r)-\omega n_{\sigma_1}\cdots n_{\sigma_{i-1}}>0.
$$
Now suppose that $i\ge 2$ and $r<i$.
We have that
$$
\begin{array}{l}
(\overline n_1a_0+\omega a_1+\ldots+\omega n_1\cdots n_{r-1}a_r-\omega n_1(\nu,R)\cdots n_{i-1})\beta_{\sigma_0(1)}(1)\\
\ge \beta_{\sigma_{i-1}(1)}(1)-a_2\beta_{\sigma_1(1)}(1)-\ldots-a_r\beta_{\sigma_{r-1}(1)}(1)>0
\end{array}
$$
since $n_{\sigma_j(1)}(1)=n_{\sigma_{j+1}}$ for all $j$, and so $n_{\sigma_{j+1}}\beta_{\sigma_j(1)}(1)<\beta_{\sigma_{j+1}(1)}(1)$ for all $j$. 

Now suppose that $i=1$. As in the proof for the case $i\ge 2$ we have that $t-\lambda>0$ if $r\ge 1$, so suppose that $i=1$ and $r=0$. Then $\overline n_{\sigma_1}\beta_{\sigma_1}=\omega \beta_{\sigma_0}$.
From our assumption $a_0\nu(P_0)\ge \nu(P_1)$ we obtain $t-\lambda=\overline n_{\sigma_1}a_0-\omega\ge 0$ with equality if and  only if $a_0=\omega=\overline n_{\sigma_1}=1$ since ${\rm gcd}(\omega,\overline n_{\sigma_1})=1$.

Now suppose $i=0$. As in the previous cases, we have $t-\lambda>0$ if $r>1$ and $t-\lambda>0$ if $r=1$ except possibly if  $P_0^{a_0}\cdots P_r^{a_r}=P_1^{a_1}$. We then have that $\nu(P_{\sigma_1}^{a_1})>\nu(P_{\sigma_0})$, and so
$$
a_1\frac{\beta_{\sigma_1}}{\beta_{\sigma_0}}>1.
$$
Since
$$
\frac{\beta_{\sigma_1}}{\beta_{\sigma_0}}=\frac{\omega}{\overline n_{\sigma_1}},
$$
we have that $t-\lambda=\omega a_1-\overline n_{\sigma_1}>0$.

\end{proof}

\begin{Lemma}\label{Lemma2} Let notation be as in Lemma \ref{Lemma1}. Suppose that $f\in R$, with $\nu(f)=\nu(P_{\sigma_i})$ for some $i\ge 0$, and that $f$ has an expression of the form of Theorem \ref{TheoremG2}, 
$$
f=cP_{\sigma_i}+\sum_{j=1}^sc_iP_{\sigma_0}^{a_0(j)}P_{\sigma_1}^{a_1(j)}\cdots P_{\sigma_r}^{a_r(j)}+h
$$
where $s,r\in\NN$, $c, c_j$  are units in $R$, with $0\le a_k(j)<n_k$ for $1\le k\le r$ for $1\le j\le s$, 
$$
\nu(f)=\nu(P_{\sigma_i})\le \nu(P_{\sigma_0}^{a_0(j)}P_{\sigma_1}^{a_1(j)}\cdots P_{\sigma_r}^{a_r(j)})
$$
for $1\le j\le s$, $a_k(j)=0$ for $k\ge i$ if $\nu(f)=\nu(P_{\sigma_0}^{a_p(j)}\cdots P_{\sigma_r}^{a_r(j)})$
and $h\in  m_R^n$ with $n>\nu(f)$.
 Then 
${\rm st}_{R_{\sigma_1}}(f)$ is a unit in $R_{\sigma_1}$ if $i=0$ or 1 and if $i>1$,   there exists a unit $\overline c$ in $R_{\sigma_1}$ and $\Omega\in R_{\sigma_1}$ such that
$$
{\rm st}_{R_{\sigma_1}}(f)=\overline cP_{\sigma_{i-1}(1)}(1)+x_1\Omega
$$
with $\nu({\rm st}_{R_{\sigma_1}}(f))=\nu(P_{\sigma_{i-1}(1)}(1))$ and $\nu(P_{\sigma_{i-1}(1)}(1))\le\nu(x_1\Omega)$.
\end{Lemma}

\begin{proof} Let
$$
\lambda=\left\{\begin{array}{ll}
\overline n_1&\mbox{ if }i=0\\
\omega&\mbox{ if }i=1\\
\omega n_{\sigma_1}\cdots n_{\sigma_{r-1}}&\mbox{ if }i\ge 2
\end{array}\right.
$$

Then
$$
f=cH_i+\sum_{j=1}^sc_j(y_1)^{aa_0(j)+ba_1(j)}P_{\sigma_0(1)}(1)^{t_j}P_{\sigma_1}(1)^{a_2(j)}\cdots P_{\sigma_{r-1}(1)}(1)^{a_r(j)}
+P_{\sigma_0(1)}(1)^th'
$$
with 
$$
H_i=\left\{
\begin{array}{ll}
(y_1)^aP_{\sigma_0(1)}(1)^{\overline n_1}&\mbox{ if }i=0\\
(y_1)^bP_{\sigma_0(1)}(1)^{\omega}&\mbox{ if }i=1\\
P_{\sigma_0(1)}(1)^{\omega n_1\cdots n_{i-1}}P_{\sigma_{i-1}(1)}(1)&\mbox{ if }i\ge 2
\end{array}\right.
$$
and
$$
t_j=\overline n_1a_0(j)+\omega a_1(j)+\omega n_{\sigma_1}a_2+\cdots+\omega n_{\sigma_1}\cdots n_{\sigma_{r-1}}a_r(j)
$$
for $1\le j\le s$, $t>\lambda$ and $h'\in R_{\sigma_1}$.
By Lemma \ref{Lemma1}, if $i\ge 2$ or $i=0$, we have that $t_j>\lambda$ for all $j$. Thus $f=P_{\sigma_0(1)}(1)^{\lambda}\overline f$
where
$$
\overline f=cG_i+\sum_{j=1}^sc_jP_{\sigma_0(1)}(1)^{t_j-\lambda}P_{\sigma_1(1)}(1)^{a_2(j)}\cdots P_{\sigma_{r-1}(1)}(1)^{a_r(j)}
+P_{\sigma_0(1)}(1)^{t-\lambda}h'
$$
with 
$$
G_i=\left\{\begin{array}{ll}
(y_1)^a&\mbox{ if }i=0\\
(y_1)^b&\mbox{ if }i=1\\
P_{\sigma_{i-1}(1)}(1)&\mbox{ if }i\ge 2
\end{array}\right.
$$
is a strict transform $\overline f={\rm st}_{R_1}(f)$ of $f$ in $R_1$.

If $i=1$, then by Lemma \ref{Lemma1}, $t_j>\lambda$ for all $j$, except possibly for a single term (that we can assume is $t_1$) which is $P_{\sigma_0}$, and we have that $\omega=\overline n_{\sigma_1}=1$. In this case $t_1=\lambda$. Then
$$
\left[\frac{P_{\sigma_1}}{P_{\sigma_0}}\right]=\alpha_{\sigma_1}(\nu,R)\in V_{\nu}/m_{\nu}
$$
which has degree $d_{\sigma_1}(\nu,R)=n_{\sigma_1}>1$ over $R/m_R$. By (\ref{eqZ1}), $x=x_1$, $y=x_1y_1$ and 
$$
f=x_1[c+c_1y_1+x_1\Omega]
$$
with $\Omega\in R_{\sigma_1}$. We have that $c+c_1y_1$ is a unit in $R_{\sigma_1}$ since 
$$
[y_1]=\left[\frac{P_{\sigma_0}}{P_{\sigma_1}}\right]\not\in R/m_R.
$$

\end{proof}

\section{Finite generation implies no defect}

Suppose that $R$ is a two dimensional regular  local ring of $K$ and $S$ is a two dimensional regular  local ring  such that  $S$ dominates $R$ 
Let $K$ be the quotient field of $R$ and $K^*$ be the quotient field of $S$. Suppose that $K\rightarrow K^*$ is a finite separable field extension. Suppose that $\nu^*$ is a non discrete rational rank 1 valuation of $K^*$ 
such that $V_{\nu^*}/m_{\nu^*}$ is algebraic over $S/m_S$ and that  $\nu^*$ dominates $S$. Then we have a natural graded inclusion ${\rm gr}_{\nu}(R)\rightarrow {\rm gr}_{\nu^*}(S)$, so that for $f\in R$, we have that ${\rm in}_{\nu}(f)={\rm in}_{\nu^*}(f)$. Let $\nu=\nu^*|K$. Let $L=V_{\nu^*}/m_{\nu^*}$. Suppose that 
 ${\rm gr}_{\nu^*}(S)$ is a finitely generated  ${\rm gr}_{\nu}(R)$-algebra.

Let $x,y$ be regular parameters in $R$, with associated generating sequence to $\nu$, $P_0=P_0(\nu,R)=x,P_1=P_1(\nu,R)=y,P_2=P_2(\nu,R),\ldots$  in $R$ as constructed in Theorem \ref{Theorem1*}, with $U_i=U_i(\nu,R)$, $\beta_i=\beta_i(\nu,R)=\nu(P_i)$, $\gamma_i=\alpha_i(\nu,R)$, $m_i=m_i(\nu,R)$, $\overline m_i=\overline m_i(\nu,R)$, $d_i=d_i(\nu,R)$ and $\sigma_i=\sigma_i(\nu,R)$ defined as in Section \ref{SecGen}.

Let $u,v$ be regular parameters in $S$, with associated generating sequence to $\nu^*$, $Q_0=P_0(\nu^*,S)=u,Q_1=P_1(\nu^*,S)=v,Q_2=P_2(\nu^*,S),\ldots$  in $S$ as constructed in Theorem \ref{Theorem1*}, with $V_i=U_i(\nu^*,S)$, $\gamma_i=\beta_i(\nu^*,S)=\nu^*(Q_i)$, $\delta_i=\alpha_i(\nu^*,S)$, $n_i=n_i(\nu^*,S)$, $\overline n_i=\overline n_i(\nu^*,S)$, $e_i=\alpha_i(\nu^*,S)$ and $\tau_i=\sigma_i(\nu^*,S)$ defined as in Section \ref{SecGen}.

With our assumption that ${\rm gr}_{\nu^*}(S)$ is a finitely generated  ${\rm gr}_{\nu}(R)$-algebra, we have that for all sufficiently large $l$, 
\begin{equation}\label{eq6}
{\rm gr}_{\nu^*}(S)={\rm gr}_{\nu}(R)[{\rm in}_{\nu^*}Q_{\tau_0},\ldots,{\rm in}_{\nu^*}Q_{\tau_l}].
\end{equation}

\begin{Proposition}\label{Prop3} With our assumption that ${\rm gr}_{\nu^*}(S)$ is a finitely generated ${\rm gr}_{\nu}(R)$-algebra,  there exist integers $s>1$ and $r>1$ such that for all $j\ge 0$,

$$
\beta_{\sigma_{r+j}}=\gamma_{\tau_{s+j}},
\overline m_{\sigma_{r+j}}=\overline n_{\tau_{s+j}},
d_{\sigma_{r+j}}=e_{\tau_{s+j}},
m_{\sigma_{r+j}}=n_{\tau_{s+j}},
$$
$$
G(\beta_{\sigma_0},\ldots,\beta_{\sigma_{r+j}})\subset G(\gamma_{\tau_0},\ldots,\gamma_{\tau_{s+j}}),
$$ 
$$
 [G(\gamma_{\tau_0},\ldots,\gamma_{\tau_{s+j}}):G(\beta_{\sigma_0},\ldots,\beta_{\sigma_{r+j}})]=e(\nu^*/\nu),
$$
$$
R/m_R[\delta_{\sigma_1},\ldots,\delta_{\sigma_{r+j}}]\subset S/m_S[\epsilon_{\tau_1},\ldots,\epsilon_{\tau_{s+j}}]
$$
and
$$
[S/m_S[\epsilon_{\tau_1},\ldots,\epsilon_{\tau_{s+j}}]:R/m_R[\delta_{\sigma_1},\ldots,\delta_{\sigma_{r+j}}]]=f(\nu^*/\nu).
$$

\end{Proposition}

\begin{proof} Let $l$ be as in (\ref{eq6}). For $s\ge l$, define the sub algebra $A_{\tau_s}$ of ${\rm gr}_{\nu^*}(S)$ by
$$
A_{\tau_s}=S/m_S[{\rm in}_{\nu^*}Q_{\tau_0},\ldots,{\rm in}_{\nu^*}Q_{\tau_s}].
$$
For $s\ge l$, let
$$
r_s=\max\{j\mid {\rm in}_{\nu^*}P_{\sigma_j}\in A_{\tau_s}\},
$$
$$
\lambda_s=[G(\gamma_{\tau_0},\ldots,\gamma_{\tau_s}):G(\beta_{\sigma_0},\ldots,\beta_{\sigma_{r_s}})],
$$
and
$$
\chi_s=[S/m_S[\epsilon_{\tau_0},\ldots,\epsilon_{\tau_s}]:R/m_R[\delta_{\sigma_0},\ldots,\delta_{\sigma_{r_s}}]].
$$
To simplify notation, we will write $r=r_s$. 

We  will now show that $\beta_{\sigma_{r+1}}=\gamma_{\tau_{s+1}}$. 
Suppose that $\beta_{\sigma_{r+1}}>\gamma_{\tau_{s+1}}$.  We have that 
$$
{\rm in}_{\nu^*}Q_{\tau_{s+1}}\in {\rm gr}_{\nu}(R)[{\rm in}_{\nu^*}Q_{\tau_0},\ldots,{\rm in}_{\nu^*}Q_{\tau_s}].
$$
Since
$$
\beta_{\sigma_{r+1}}<\beta_{\sigma_{r+2}}<\cdots
$$
we then have that
${\rm in}_{\nu^*}Q_{\tau_{s+1}}\in A_{\tau_s}$ which is impossible. Thus $\beta_{\sigma_{r+1}}\le \gamma_{\tau_{s+1}}$.
If $\beta_{\sigma_{r+1}}<\gamma_{\tau_{s+1}}$, then since
$$
\gamma_{\tau_{s+1}}<\gamma_{\tau_{s+2}}<\cdots
$$
and ${\rm in}_{\nu^*}P_{\sigma_{r+1}}\in {\rm gr}_{\nu^*}(S)$, we have that ${\rm in}_{\nu^*}P_{\sigma_{r+1}}\in A_{\tau_s}$, which is impossible. Thus
$\beta_{\sigma_{r+1}}=\gamma_{\tau_{s+1}}$.

We will now establish that either we have a reduction $\lambda_{s+1}<\lambda_s$ or
\begin{equation}\label{eq22}
\lambda_{s+1}=\lambda_s,\,\, \beta_{\sigma_{r+1}}=\gamma_{\tau_{s+1}}\mbox{ and } \overline m_{\sigma_{r+1}}=\overline n_{\tau_{s+1}}.
\end{equation}
Let $\omega$ be a generator of the group $G(\gamma_{\tau_1},\ldots,\gamma_{\tau_s})$, so that $G(\gamma_{\tau_1},\ldots,\gamma_{\tau_s})=\ZZ\omega$. We have that
$$
G(\gamma_{\tau_0},\ldots,\gamma_{\tau_{s+1}})=\frac{1}{\overline n_{\tau_{s+1}}}\ZZ\omega
$$
and
$$
G(\beta_{\sigma_0},\ldots,\beta_{\sigma_{r+1}})=\frac{1}{\overline m_{\sigma_{r+1}}}\ZZ(\lambda_s\omega).
$$
There exists a positive integer $f$ with $\mbox{gcd}(f,\overline n_{\tau_{s+1}})=1$ such that 
$$
\gamma_{\tau_{s+1}}=\frac{f}{\overline n_{\tau_{s+1}}}\omega
$$
 There exists a positive integer $g$ with $\mbox{gcd}(g,\overline m_{\sigma_{r+1}})=1$ such that
 $$
 \beta_{\sigma_{r+1}}=\frac{g}{\overline m_{\sigma_{r+1}}}\lambda_s\omega.
 $$
 Since $\beta_{\sigma_{r+1}}=\gamma_{\tau_{s+1}}$, we have 
 $$
 g\lambda_s \overline n_{\tau_{s+1}}=f \overline m_{\sigma_{r+1}}.
 $$
 Thus $\overline n_{\tau_{s+1}}$ divides $\overline m_{\sigma_{r+1}}$ and $\overline m_{\sigma_{r+1}}$ divides $\lambda_s \overline n_{\tau_{s+1}}$, so that
 $$
 a=\frac{\overline m_{\sigma_{r+1}}}{\overline n_{\tau_{s+1}}}
 $$
 is a positive integer and defining
 $$
 \overline\lambda=\frac{\lambda_s}{a},
 $$
 we have that $\overline\lambda$ is a positive integer with
 $$
 \frac{\lambda_s}{\overline m_{\sigma_{r+1}}}=\frac{\overline \lambda}{\overline n_{\tau_{s+1}}}
 $$
 and 
 $$
 \overline\lambda =[G(\gamma_{\tau_0},\ldots,\gamma_{\tau_{s+1}}):G(\beta_{\sigma_0},\ldots,\beta_{\sigma_{r+1}})].
 $$
 Since $\lambda_{s+1}\le \overline\lambda$,  either $\lambda_{s+1}<\lambda_s$ or $\lambda_{s+1}=\lambda_s$ and $\overline m_{\sigma_{r+1}}=\overline n_{\tau_{s+1}}$.
 
 We will now suppose that $s$ is sufficiently large that (\ref{eq22}) holds.  Since 
 $$
 {\rm in}_{\nu^*}Q_{\tau_{s+1}}\in {\rm gr}_{\nu^*}(S)={\rm gr}_{\nu}(R)[{\rm in}_{\nu^*}Q_{\tau_0},\ldots,{\rm in}_{\nu^*}Q_{\tau_s}],
 $$
 if $\overline n_{\tau_{s+1}}>1$ we have an expression
 \begin{equation}\label{eqN2}
 {\rm in}_{\nu^*}P_{\sigma_{r+1}}={\rm in}_{\nu^*}(\alpha){\rm in}_{\nu^*}Q_{\tau_{s+1}}
 \end{equation}
 in $\mathcal P_{\gamma_{\tau_{s+1}}}(S)/\mathcal P^+_{\gamma_{\tau_{s+1}}}(S)$ with $\alpha$ a unit in $S$  and if $\overline n_{\tau_{s+1}}=1$, 
 since ${\rm in}_{\nu^*}P_{\sigma_{r+1}}\not\in A_{\tau_s}$, we have an expression 
 \begin{equation}\label{eqN3}
 {\rm in}_{\nu^*}P_{\sigma_{r+1}}={\rm in}_{\nu^*}(\alpha){\rm in}_{\nu^*}Q_{\tau_{s+1}}
 +\sum {\rm in}_{\nu^*}(\alpha_J)({\rm in}_{\nu^*}Q_{\tau_0})^{j_0}\cdots ({\rm in}_{\nu^*}Q_{\tau_s})^{j_s} 
 \end{equation}
 in $\mathcal P_{\gamma_{\tau_{s+1}}}(S)/\mathcal P^+_{\gamma_{\tau_{s+1}}}(S)$ with $\alpha$ a unit in $S$ and the sum is over certain $J=(j_i,\ldots,j_s)\in \NN^s$ such that the  $\alpha_J$  are units in $S$,   and the terms
${\rm in}_{\nu^*}Q_{\tau_{s+1}}$ and the $({\rm in}_{\nu^*}Q_{\tau_0})^{j_0}\cdots ({\rm in}_{\nu^*}Q_{\tau_s})^{j_s} $ are linearly independent over $S/m_S$. 
 
 The monomial $U_{\sigma_{r+1}}$ in $P_{\sigma_0},\ldots, P_{\sigma_r}$ and the monomial $V_{\tau_{s+1}}$ in $Q_{\tau_0},\ldots, Q_{\tau_s}$ both have the value
 $\overline n_{\tau_{s+1}}\gamma_{\tau_{s+1}}=\overline m_{\sigma_{r+1}}\beta_{\sigma_{r+1}}$, and satisfy
 $$
 \epsilon_{\tau_{s+1}}=\left[\frac{Q_{\tau_{s+1}}^{\overline n_{\tau_{s+1}}}}{V_{\tau_{s+1}}}\right]
 $$
 and
 $$
 \delta_{\sigma_{r+1}}=\left[\frac{P_{\sigma_{r+1}}^{\overline n_{\tau_{s+1}}}}{U_{\sigma_{r+1}}}\right].
 $$
 Since $U_{\sigma_{r+1}}, V_{\tau_{s+1}}\in A_{\tau_s}$ and by (\ref{eqZ51}) and 2) of Theorem \ref{Theorem1*}, 
 we have that
 $$
 \left[\frac{V_{\tau_{s+1}}}{U_{\sigma_{r+1}}}\right]\in S/m_S[\epsilon_{\tau_1},\ldots,\epsilon_{\tau_s}].
 $$
 If $\overline n_{\tau_{s+1}}>1$, then by (\ref{eqN2}), we have
 $$
 \left[\frac{P_{\sigma_{r+1}}^{\overline n_{\tau_{s+1}}}}{U_{\sigma_{r+1}}}\right]
 =\left[\frac{V_{\tau_{s+1}}}{U_{\sigma_{r+1}}}\right]
 \left(\left[\alpha\right]^{\overline n_{\tau_{s+1}}}\left[\frac{Q_{\tau_{s+1}}^{\overline n_{\tau_{s+1}}}}{V_{\tau_{s+1}}}\right]\right)
 $$
 in $L=V_{\nu^*}/m_{\nu^*}$, and if $\overline n_{\tau_{s+1}}=1$, then by (\ref{eqN3}), we have
 $$
 \left[\frac{P_{\sigma_{r+1}}}{U_{\sigma_{r+1}}}\right]
 =\left[\frac{V_{\tau_{s+1}}}{U_{\sigma_{r+1}}}\right]
 \left(\left[\alpha\right]\left[\frac{Q_{\tau_{s+1}}}{V_{\tau_{s+1}}}\right]
 +\sum\left[\alpha_J\right]\left[\frac{Q_{\tau_0}^{j_0}\cdots Q_{\tau_s}^{j_s}}{V_{\tau_{s+1}}}
 \right] \right) .
 $$ 
 Thus by equation (\ref{eqZ51}),
 \begin{equation}\label{eqN4}
 S/m_S[\epsilon_{\tau_1},\ldots,\epsilon_{\tau_s}][\epsilon_{\tau_{s+1}}]=S/m_S[\epsilon_{\tau_1},\ldots,\epsilon_{\tau_s}][\delta_{\sigma_{r+1}}].
 \end{equation}
 We have a commutative diagram
 $$
 \begin{array}{ccccc}
S/m_S[\epsilon_{\tau_1},\ldots,\epsilon_{\tau_s}]&\rightarrow& S/m_S[\epsilon_{\tau_1},\ldots,\epsilon_{\tau_s},\epsilon_{\tau_{s+1}}]&=& S/m_S[\epsilon_{\tau_1},\ldots,\epsilon_{\tau_s}][\delta_{\sigma_{r+1}}]\\
\uparrow&&\uparrow\\
R/m_R[\delta_{\sigma_1},\ldots,\delta_{\sigma_r}]&\rightarrow&R/m_R[\delta_{\sigma_1},\ldots,\delta_{\sigma_r}][\delta_{\sigma_{r+1}}].
\end{array}
$$
Let
$$
\overline\chi=[S/m_S[\epsilon_{\tau_1},\ldots,\epsilon_{\tau_s},\epsilon_{\tau_{s+1}}]:R/m_R[\delta_{\sigma_1},\ldots,\delta_{\sigma_r},\delta_{\sigma_{r+1}}]].
$$
Since 
$$
S/m_S[\epsilon_{\tau_1},\ldots,\epsilon_{\tau_s},\epsilon_{\tau_{s+1}}]= S/m_S[\epsilon_{\tau_1},\ldots,\epsilon_{\tau_s}][\delta_{\sigma_{r+1}}],
$$
we have that $e_{\tau_{s+1}}|d_{\sigma_{r+1}}$. Further, 
$$
\frac{d_{\sigma_{r+1}}}{e_{\tau_{s+1}}}\overline\chi=\chi_s,
$$
whence $\overline\chi\le\chi_s$. Thus $\chi_{s+1}\le\chi_s$ and if $\chi_{s+1}=\chi_s$, then  $d_{\sigma_{r+1}}=e_{\tau_{s+1}}$ and $r_{s+1}=r_s+1$  since $P_{\sigma_{r+2}}\in A_{\tau_{s+1}}$ implies $\lambda_{s+1}<\lambda_s$ or $\chi_{s+1}<\chi_s$.

We may thus choose $s$ sufficiently large that there exists an integer $r>1$ such that for all $j\ge 0$,

$$
\beta_{\sigma_{r+j}}=\gamma_{\tau_{s+j}},
\overline m_{\sigma_{r+j}}=\overline n_{\tau_{s+j}},
d_{\sigma_{r+j}}=e_{\tau_{s+j}},
m_{\sigma_{r+j}}=n_{\tau_{s+j}},
$$
$$
G(\beta_{\sigma_0},\ldots,\beta_{\sigma_{r+j}})\subset G(\gamma_{\tau_0},\ldots,\gamma_{\tau_{s+j}}),
$$ 
there is a constant $\lambda$ (which does not depend on $j$) such that 
$$
 [G(\gamma_{\tau_0},\ldots,\gamma_{\tau_{s+j}}):G(\beta_{\sigma_0},\ldots,\beta_{\sigma_{r+j}})]=\lambda
 $$
$$
R/m_R[\delta_{\sigma_1},\ldots,\delta_{\sigma_{r+j}}]\subset S/m_S[\epsilon_{\tau_1},\ldots,\epsilon_{\tau_{s+j}}]
$$
and there is a constant $\chi$ (which does not depend on $j$) such that 
$$
[S/m_S[\epsilon_{\tau_1},\ldots,\epsilon_{\tau_{s+j}}]:R/m_R[\delta_{\sigma_1},\ldots,\delta_{\sigma_{r+j}}]]=\chi.
$$
Then
$$
\Phi_{\nu^*}=\cup_{j\ge 1}\frac{1}{\overline n_{\tau_{s+1}}\cdots\overline n_{\tau_{s+j}}}\ZZ\omega
$$
where $G(\gamma_{\tau_0},\ldots,\gamma_{\tau_s})=\ZZ\omega$,
and 
$$
\Phi_{\nu}=\cup_{j\ge 1}\frac{1}{\overline m_{\sigma_{r+1}}\cdots\overline m_{\sigma_{r+j}}}\lambda\ZZ\omega= \cup_{j\ge 1}\frac{1}{\overline n_{\tau_{s+1}}\cdots\overline n_{\tau_{s+j}}}\lambda\ZZ\omega
$$
so that
$$
\lambda=[\Phi_{\nu^*}:\Phi_{\nu}]=e(\nu^*/\nu).
$$
For $i\ge 0$, let $K_i=R/m_R[\delta_{\sigma_1},\ldots,\delta_{\sigma_{r+i}}]$ and $M_i=S/m_S[\epsilon_{\tau_1},\ldots,\epsilon_{\tau_{s+i}}]$.
We have that $M_{i+1}=M_i[\delta_{\sigma_{r+i+1}}]$ for $i\ge 0$ and $\chi=[M_i:K_i]$ for all $i$. Further,
$$
\cup_{i=0}^{\infty}M_i=V_{\nu^*}/m_{\nu^*}\mbox{ and }\cup_{i=0}^{\infty}K_i=V_{\nu}/m_{\nu}.
$$
 Thus if $g_1,\ldots,g_{\lambda}\in M_0$ form a basis of $M_0$ as a $K_0$-vector space, then $g_1,\ldots,g_{\lambda}$ form a basis of $M_i$ as a $K_i$-vector space for all $i\ge 0$. Thus
$$
\chi=[V_{\nu^*}/m_{\nu^*}:V_{\nu}/m_{\nu}]=f(\nu^*/\nu).
$$
\end{proof}

Let $r$ and $s$ be as in the conclusions of Proposition \ref{Prop3}.
There exists $\tau_t$ with $t\ge s$ such that we have a commutative diagram of inclusions of regular local rings (with the notation introduced  in  Section \ref{SecGen})
$$
\begin{array}{ccc}
R_{\sigma_r}&\rightarrow &S_{\tau_t}\\
\uparrow&&\uparrow\\
R&\rightarrow&S.
\end{array}
$$
After possibly increasing $s$ and $r$, we may assume that $R'\subset R_{\sigma_r}$, where $R'$ is the local ring of the conclusions of Proposition \ref{Prop15}.
Recall that $R$ has regular parameters $x=P_0$, $y=P_1$ and $S$ has regular parameters $u=Q_0$, $v=Q_1$, $R_{\sigma_r}$ has regular parameters $x_{\sigma_r}$, $y_{\sigma_r}$ such that
$$
x=\delta x_{\sigma_r}^{\overline m_{\sigma_1}\cdots \overline m_{\sigma_r}},\,\, y_{\sigma_r}={\rm st}_{R_{\sigma_r}}P_{\sigma_{r+1}}
$$
where $\delta$ is a unit in $R_{\sigma_r}$ and 
$S_{\tau_t}$ has regular parameters $u_{\tau_t}$, $v_{\tau_t}$ such that
$$
u=\epsilon u_{\tau_t}^{\overline n_{\tau_1}\cdots \overline n_{\tau_t}},\,\,v_{\tau_t}={\rm st}_{S_{\tau_t}}Q_{\tau_{t+1}}
$$
where $\epsilon$ is a unit in $S_{\tau_t}$. We may choose $t\gg 0$ so that we we have an expression 
\begin{equation}\label{eq7}
x_{\sigma_r}=\phi u_{\tau_t}^{\lambda}
\end{equation}
for some positive integer $\lambda$ where $\phi$ is a unit in $S_{\tau_t}$, since $\cup_{t=0}^{\infty}S_{\tau_t}=V_{\nu^*}$.

We have expressions $P_i=\psi_i x_{\sigma_r}^{c_i}$ in $R_{\sigma_r}$ where $\psi_i$ are units in $R_{\sigma_r}$ for $i\le \sigma_r$ so that 
$P_i=\psi_i^* u_{\tau_t}^{c_i\lambda}$ in $S_{\tau_t}$ where $\psi_i^*$ are units in $S_{\tau_t}$ for $i\le \sigma_r$ by (\ref{eq7}).

\begin{Lemma}\label{Lemma3} For $j\ge 1$ we have
$$
{\rm st}_{R_{\sigma_r}}(P_{\sigma_{r+j}})=u_{\tau_t}^{\lambda_j}{\rm st}_{S_{\tau_t}}(P_{\sigma_{r+j}})
$$
for some $\lambda_j\in \NN$,  where we regard $P_{\sigma_{r+j}}$ as an element of $R$ on the left hand side of the equation and regard $P_{\sigma_{r+j}}$ as an element of $S$ on the right hand side.  
\end{Lemma}

\begin{proof} Using (\ref{eq7}), we have 
$$
P_{\sigma_{r+j}}={\rm st}_{R_{\sigma_r}}(P_{\sigma_{r+j}})x_{\sigma_r}^{f_j}
={\rm st}_{R_{\sigma_r}}(P_{\sigma_{r+j}})u_{\tau_t}^{\lambda f_j}\phi^{f_j}
$$
where $f_j\in \NN$. 
Viewing $P_{\sigma_{r+j}}$ as an element of $S$, we have that
$$
P_{\sigma_{r+j}}={\rm st}_{S_{\tau_t}}(P_{\sigma_{r+j}})u_{\tau_t}^{g_j}
$$
for some $g_j\in \NN$. Since $u_{\tau_t}\not\,\mid {\rm st}_{S_{\tau_t}}(P_{\sigma_{r+j}})$, we have that $f_j\lambda \le g_j$ and so $\lambda_j=g_j-f_j\lambda\ge 0$.

\end{proof}

By induction in the sequence of quadratic transforms above $R$ and $S$  in Lemma \ref{Lemma2}, and since $\nu^*(P_{\sigma_{r+j}})=\beta_{\sigma_{r+j}}=\gamma_{\tau_{s+j}}$ by Proposition \ref{Prop3}, we have by (\ref{eqN2}) and (\ref{eqN3}) an expression 
\begin{equation}\label{eqN60}
{\rm st}_{S_{\tau_t}}(P_{\sigma_{r+j}})=c{\rm st}_{S_{\tau_t}}(Q_{\tau_{s+j}})+u_{\tau_t}\Omega
\end{equation}
with $c\in S_{\tau_t}$ a unit, $\Omega\in S_{\tau_t}$ and $\nu^*(u_{\tau_t}\Omega)\ge\nu^*({\rm st}_{S_{\tau_t}}(Q_{\tau_{s+j}}))$ if $s+j>t$ and
\begin{equation}\label{eqN61}
S_{\tau_t}(P_{\sigma_{r+j}})\mbox{ is a unit in $S_{\tau_t}$}
\end{equation}
if $s+j\le t$.
Thus
$P_{\sigma_{r+j}}=u_{\tau_t}^{d_j}\overline\phi_j$ in $S_{\tau_t}$ where $d_j$ is a positive integer and $\overline\phi_j$ is a unit in $S_{\tau_t}$ if $s+j\le t$.

Suppose $s<t$. Then
$$
y_{\sigma_r}={\rm st}_{R_{\sigma_r}}(P_{\sigma_{r+1}})=\tilde\phi u_{\tau_t}^{h}
$$
where $\tilde\phi$ is a unit in $S_{\tau_{t}}$ and $h$ is a positive integer.
As shown in equation (\ref{eq8}) of Section \ref{SecGen},  
$$
R_{\sigma_{r+1}}=R_{\sigma_r}[\overline x_1,\overline y_1]_{m_{\nu}\cap R_{\sigma_r}[\overline x_1,\overline y_1]}
$$
 where
$$
\overline x_1=(x_{\sigma_r}^{b}y_{\sigma_r}^{-a})^{\epsilon},\,\,\, \overline y_1=(x_{\sigma_r}^{-\omega}y_{\sigma_r}^{m_{\sigma_r}})^{\epsilon}
$$
with $\epsilon=m_{\sigma_r}b-\omega a=\pm 1$,
$\nu(\overline x_1)>0$ and $\nu(\overline y_1)=0$. 
Substituting 
$$
x_{\sigma_r}=\phi u_{\tau_t}^{\lambda}\mbox{ and }y_{\sigma_1}=\tilde\phi u_{\tau_t}^h
$$
 we see that
$R_{\sigma_{r+1}}$ is dominated by $S_{\tau_t}$.
We thus have  a factorization
 $$
 R_{\sigma_r}\rightarrow R_{\sigma_{r+1}}\rightarrow S_{\tau_t}
 $$
with $x_{\sigma_{r+1}}=\overline x_1=\hat\phi u_{\tau_t}^{\lambda'}$ where $\hat\phi$ is a unit in $S_{\tau_t}$ and $\lambda'$ is a positive integer. We may thus replace $s$ with $s+1$, $r$ with $r+1$ and $R_{\sigma_r}$ with $R_{\sigma_{r+1}}$. 

Iterating this argument, we may assume that $s=t$ (with $r=r_s$) so that by Lemma \ref{Lemma3}, (\ref{eqN61}) and (\ref{eqN60}),
$$
y_{\sigma_r}={\rm st}_{R_{\sigma_r}}(P_{\sigma_{r+1}})=u_{\tau_s}^{\mu}{\rm st}_{S_{\tau_s}}(P_{\sigma_{r+1}})
$$
where 
$$
{\rm st}_{S_{\tau_s}}(P_{\sigma_{r+1}})=c\,{\rm st}_{S_{\tau_s}}(Q_{\tau_{s+1}})+u_{\tau_s}\Omega
$$
with $c$ a unit in $S_{\tau_s}$  and $\Omega\in S_{\tau_s}$. Thus by (\ref{eq7}), we have an expression
$$
x_{\sigma_r}= \phi u_{\tau_s}^{\lambda},\,\, y_{\sigma_r}=\overline\epsilon u_{\tau_s}^{\alpha}(v_{\tau_s}+u_{\tau_s}\Omega)
$$
where $\lambda$ is a positive integer, $\alpha\in \NN$, $\phi$ and $\overline\epsilon$ are units in $S_{\tau_s}$ and $\Omega\in S_{\tau_s}$.

We have that $\nu^*(x_{\sigma_r})=\lambda\nu^*(u_{\tau_s})$,
$$
\begin{array}{lll}
\nu(x_{\sigma_r})\ZZ&=&G(\nu(x_{\sigma_r}))=G(\beta_{\sigma_0},\ldots,\beta_{\sigma_r})\mbox{ and }\\
\nu^*(u_{\tau_s})\ZZ&=&G(\nu^*(u_{\tau_s}))=G(\gamma_{\tau_0},\ldots,\gamma_{\tau_s}).
\end{array}
$$
Thus 
$$
\lambda=[G(\gamma_{\tau_0},\ldots,\gamma_{\tau_s}):G(\beta_{\sigma_0},\ldots,\beta_{\sigma_r})]=e(\nu^*/\nu)
$$
by Proposition \ref{Prop3}.

By Theorem \ref{birat}, we have that
$$
R_{\sigma_r}/m_{R_{\sigma_r}}=R/m_R[\delta_{\sigma_1},\ldots,\delta_{\sigma_r}]\mbox{ and }S_{\tau_s}/m_{S_{\tau_s}}=S/m_S[\epsilon_{\tau_1},\ldots,\epsilon_{\tau_s}].
$$
Thus
$$
[S_{\tau_s}/m_{S_{\tau_s}}:R_{\sigma_r}/m_{R_{\sigma_r}}]=f(\nu^*/\nu)
$$
by Proposition \ref{Prop3}.

Since the ring $R'$ of Proposition \ref{Prop15} is contained in $R_{\sigma_r}$ by our construction, we have  by Proposition \ref{Prop15} that $(K,\nu)\rightarrow (K^*,\nu^*)$ is without defect, completing the proofs of Proposition \ref{Prop1} and Theorem \ref{Theorem4}.

\section{non splitting and finite generation}

In this section, we will have the following assumptions. Suppose that $R$ is a 2 dimensional excellent local domain with quotient field $K$. Further suppose that $K^*$ is a finite separable extension of $K$ and $S$ is a 2 dimensional local domain with quotient field
$K^*$ such that  $S$ dominates $R$. 
Suppose that $\nu^*$ is a valuation of $K^*$ such that 
 $\nu^*$ dominates $S$. Let $\nu$ be the restriction of $\nu^*$ to $K$. 
 
 Suppose that $\nu^*$ has rational rank 1 and $\nu^*$ is not discrete.  Then $V_{\nu^*}/m_{\nu^*}$ is algebraic over $S/m_S$, by Abhyankar's inequality, Proposition 2 \cite{Ab1}. 
 
 \begin{Lemma} \label{LemmaN1} Let assumptions be as above. Then the associated graded ring ${\rm gr}_{\nu^*}(S)$ is an integral extension of ${\rm gr}_{\nu}(R)$.
 \end{Lemma}
 
 \begin{proof}  It suffices to show that ${\rm in}_{\nu^*}(f)$ is integral over ${\rm gr}_{\nu}(R)$ whenever $f\in S$. Suppose that 
 $f\in S$. There exists $n_1>0$ such that $n_1\nu^*(f)\in \Phi_{\nu}$. Let $x\in m_R$ and $\omega=\nu(x)$. Then there exists a  positive integer  $b$ and natural number $a$ such that $bn_1\nu^*(f)=a\omega$, so
 $$
 \nu^*\left(\frac{f^{bn_1}}{x^a}\right)=0.
 $$
 Let 
 $$
 \xi=\left[\frac{f^{bn_1}}{x^a}\right]\in V_{\nu^*}/m_{\nu^*},
 $$
 and let $g(t)=t^r+\overline a_{r-1}t^{r-1}+\cdots+\overline a_0$ with $\overline a_i\in R/m_R$ be the minimal polynomial of $\xi$ over $R/m_R$. Let $a_i$ be lifts of the $\overline a_i$ to $R$.
 Then
 $$
 \nu^*(f^{b_1n_1r}+a_{r-1}x^af^{bn_1(r-1)}+\cdots+a_0x^{ar})>\nu^*(f^{bn_1r})=\nu^*(a_{r-1}x^af^{bn_1(r-1)})=\cdots=\nu^*(a_0x^{ar}).
 $$
 Thus 
 $$
 {\rm in}_{\nu^*}(f)^{b_1n_1r}+{\rm in}_{\nu}(a_{r-1}x^a){\rm in }_{\nu^*}(f)^{bn_1(r-1)}+\cdots+{\rm in}_{\nu}(a_0x^{ar})=0
 $$
 in ${\rm gr}_{\nu^*}(S)$. Thus ${\rm in}_{\nu^*}(f)$ is integral over ${\rm gr}_{\nu^*}(R)$.
 \end{proof}
 
 We now establish Theorem \ref{ThmN2}. Recall (as defined after Proposition \ref{Prop1}) that $\nu^*$ does not split in $S$ if $\nu^*$ is the unique extension of $\nu$ to $K^*$ which dominates $S$.
 
 \begin{Theorem} Let assumptions be as above and suppose that $R$ and $S$ are regular local rings. Suppose that ${\rm gr}_{\nu^*}(S)$ is a finitely generated ${\rm gr}_{\nu}(R)$-algebra. Then $S$ is a localization of the integral closure of $R$ in $K^*$, the defect $\delta(\nu^*/\nu)=0$ and $\nu^*$ does not split in $S$.  \end{Theorem}

 \begin{proof} Let $s$ and $r$ be as in the conclusions of Proposition \ref{Prop3}. We will first show that $P_{\sigma_{r+j}}$ is irreducible in $\hat S$ for all $j>0$. There exists a unique extension of $\nu^*$ to the quotient field of $\hat S$ which dominates $\hat S$ (\cite{Sp}, \cite{CV1}, \cite{GAST}).  The extension is immediate since $\nu^*$ is not discrete; that is, there is no increase in value group or residue field for the extended valuation. It has the property that if $f\in \hat S$ and $\{f_i\}$ is a a Cauchy sequence in $\hat S$ which converges to $f$, then $\nu^*(f)=\nu^*(f_i)$ for all $i\gg 0$. 
 
 Suppose that $P_{\sigma_{r+j}}$ is not irreducible in $\hat S$ for some $j>0$. We will derive a contradiction. With this assumption,  $P_{\sigma_{r+j}}=fg$ with $f,g\in m_{\hat S}$. Let $\{f_i\}$ be a Cauchy sequence in $S$ which converges to $f$ and let $\{g_i\}$ be a Cauchy sequence in $S$ which converges to $g$. For $i$ sufficiently large, $f-f_i,g-g_i\in m_{\hat S}^n$ where $n$ is so large that $n\nu^*(m_{\hat S})=n\nu^*(m_S)>\nu(P_{\sigma_{r+j}})$.
 Thus $P_{\sigma_{r+j}}=f_ig_i+h$ with $h\in m_{\hat S}^n\cap S=m_S^n$, and so ${\rm in}_{\nu^*}(P_{\sigma_{r+j}})={\rm in}_{\nu^*}(f_i){\rm in}_{\nu^*}(g_i)$. Now
 $$
 \nu^*(f_i),\nu^*(g_i)<\nu(P_{\sigma_{r+j}})=\beta_{\sigma_{r+j}}=\gamma_{\tau_{s+j}}=\nu^*(Q_{\tau_{s+j}})
 $$
 so that 
 $$
 {\rm in}_{\nu^*}(f_i),{\rm in}_{\nu^*}(g_i)\in S/m_S[{\rm in }_{\nu^*}(Q_{\tau_0}),\ldots, {\rm in}_{\nu^*}(Q_{\tau_{s+j-1}})]
 $$
 which  implies 
 $$
 {\rm in}_{\nu^*}(P_{\sigma_{r+j}})\in S/m_S[{\rm in }_{\nu^*}(Q_{\tau_0}),\ldots, {\rm in}_{\nu^*}(Q_{\tau_{s+j-1}})].
 $$ 
 But then (\ref{eqN3}) implies 
 $$
 {\rm in}_{\nu^*}(Q_{\tau_{s+j}})\in S/m_S[{\rm in }_{\nu^*}(Q_{\tau_0}),\ldots, {\rm in}_{\nu^*}(Q_{\tau_{s+j-1}})]
 $$  
 which is impossible. Thus $P_{\sigma_{r+j}}$ is irreducible in $\hat S$ for all $j>0$.
 
 If $S$ is not a localization of the integral closure of $R$ in $K^*$, then by Zariski's Main Theorem (Theorem 1 of Chapter 4 \cite{R}), $m_RS=fN$ where $f\in m_S$ and $N$ is an $m_S$-primary ideal. Thus $f$ divides $P_i$ in $S$ for all $i$, which is impossible since we have shown that $P_{\sigma_{r+j}}$ is analytically irreducible in $S$ for all $j>0$; we cannot have $P_{\sigma_{r+j}}=a_jf$ where $a_j$ is a unit in $S$ for $j>0$  since 
 $\nu(P_{\sigma_{r+j}})=\nu^*(Q_{\tau_{s+j}})$ by Proposition \ref{Prop3}.

 Now suppose that $\nu^*$ is not the unique extension of $\nu$ to $K^*$ which dominates $S$.   Recall that $V_{\nu}$ is the union of all quadratic transforms above $R$ along $\nu$ and $V_{\nu^*}$ is the union of all quadratic transforms above $S$ along $\nu^*$ (Lemma 4.5 \cite{RTM}).
 
 Then for all $i\gg 0$, we have a commutative diagram
 $$
 \begin{array}{lll}
 R_{\sigma_i}&\rightarrow& T_i\\
 \uparrow&&\uparrow\\
 R&\rightarrow &T
 \end{array}
 $$
 where $T$ is the integral closure of $R$ in $K^*$, $T_i$ is the integral closure of $R_{\sigma_i}$ in $K^*$, $S=T_{\mathfrak p}$ for some maximal ideal $\mathfrak p$ in $T$ which lies over $m_R$,
 and there exist $r\ge 2$ prime ideals $\mathfrak p_1(i),\ldots,\mathfrak p_r(i)$ in $T_i$ which lie over $m_{R_{\sigma_i}}$ and whose intersection with $T$ is $\mathfrak p$. We may assume that $\mathfrak p_1(i)$ is the center of $\nu^*$. 
 
 There exists an $m_R$-primary ideal $I_i$ in $R$ such that the blow up of $I_i$ is $\gamma:X_{\sigma_i}\rightarrow \mbox{Spec}(R)$ where $X_{\sigma_i}$ is regular and $R_{\sigma_i}$ is a local ring of $X_{\sigma_i}$. Let $Z_{\sigma_i}$ be the integral closure of $X_{\sigma_i}$ in $K^*$. Let $Y_{\sigma_i}=Z_{\sigma_i}\times_{\mbox{Spec}(T)}\mbox{Spec}(S)$.
 We have a commutative diagram  of morphisms
 $$
 \begin{array}{lll}
 Y_{\sigma_i}&\stackrel{\beta}{\rightarrow}&X_{\sigma_i}\\
 \delta\downarrow&&\gamma\downarrow\\
 \mbox{Spec}(S)&\stackrel{\alpha}{\rightarrow}&\mbox{Spec}(R)
 \end{array}
 $$
 The morphism $\delta$ is projective (by Proposition II.5.5.5 \cite{EGAII} and Corollary II.6.1.11 \cite{EGAII} and it is birational, so since $Y_{\sigma_i}$ and $\mbox{Spec}(S)$ are integral, it is a blow up of an ideal $J_i$ in $S$ (Proposition III.2.3.5 \cite{EGAIII}), which we can take to be $m_S$-primary since $S$ is a regular local ring and hence factorial. Define curves $C=\mbox{Spec}(R/(P_{\sigma_i}))$ and $C'=\alpha^{-1}(C)=\mbox{Spec}(S/(P_{\sigma_i}))$. Denote the Zariski closure of a set $W$ by $\overline W$. The strict transform $C^*$ of $C'$ in $Y_{\sigma_i}$ is the Zariski closure
  \begin{equation}\label{eqN21}
  \begin{array}{lll}
  C^*&=&\overline{\delta^{-1}(C'\setminus m_S)}=\overline{\delta^{-1}\alpha^{-1}(C\setminus m_R)}
  =\overline{\beta^{-1}\gamma^{-1}(C\setminus m_R)}\\
  &=&\beta^{-1}(\overline{\gamma^{-1}(C\setminus m_R)})\mbox{ since $\beta$ is quasi finite}\\
  &=& \beta^{-1}(\tilde C)
  \end{array}
  \end{equation}
 where $\tilde C$ is the strict transform of $C$ in $X_{\sigma_i}$.  We have that $Z_{\sigma_i}\times_{X_{\sigma_i}}\mbox{Spec}(R_{\sigma_i})\cong\mbox{Spec}(T_i)$, so
 $$
 Y_{\sigma_i}\times_{X_{\sigma_i}}\mbox{Spec}(R_{\sigma_i})\cong \mbox{Spec}(T_i\otimes_TS).
 $$
 Let $x_{\sigma_i}$ be a local equation in $R_{\sigma_i}$ of the exceptional divisor of $\mbox{Spec}(R_{\sigma_i})\rightarrow \mbox{Spec}(R)$ and let $y_{\sigma_i}=\mbox{st}_{R_{\sigma_i}}(P_{\sigma_i})$. Then 
 $x_{\sigma_i},y_{\sigma_i}$ are regular parameters in $R_{\sigma_i}$. We have that
 $$
 \sqrt{m_{R_{\sigma_i}}(T_i\otimes_TS)}=\cap_{j=1}^r\mathfrak p_j(i)(T_i\otimes_TS).
 $$
 The blow up of $J_i(S/(P_{\sigma_i}))$ in $C'$ is $\overline\delta:C^*\rightarrow C'$, where $\overline\delta$ is the restriction of $\delta$ to $C^*$  Corollary II.7.15 \cite{H}). Since $y_{\sigma_i}$ is a local equation of $\tilde C$ in $R_{\sigma_i}$, we have by (\ref{eqN21})  that 
 $$
 \mathfrak p_1(i),\ldots,\mathfrak p_r(i)\in \overline\delta^{-1}(m_S)\subset C^*.
 $$
 Since $\overline\delta$ is proper and $C'$ is a curve, $C^*=\mbox{Spec}(A)$ for some excellent one dimensional  domain $A$ such that the inclusion $S/(P_{\sigma_i})\rightarrow A$ is finite (Corollary I.1.10 \cite{Mi}).
 Let $B=A\otimes_{S/(P_{\sigma_i})}\hat S/(P_{\sigma_i})$. Then
 $$
 C^*\times_{\mbox{Spec}(S/(P_{\sigma_i}))}\mbox{Spec}(\hat S/(P_{\sigma_i}))=\mbox{Spec}(B)\rightarrow \mbox{Spec}(\hat S/(P_{\sigma_i}))
 $$
 is the blow up of $J_i(\hat S/(P_{\sigma_i}))$ in $\hat S/(P_{\sigma_i})$.  The extension $\hat S/(P_{\sigma_i})\rightarrow B$ is finite since  $S/(P_{\sigma_i})\rightarrow A$ is finite.

Now assume that $S/(P_{\sigma_i})$ is analytically irreducible. Then $B$ has only one minimal prime since the blow up $\mbox{Spec}(B)\rightarrow \mbox{Spec}(\hat S/(P_{\sigma_i}))$ is birational. 
 
 Since a complete local ring is Henselian, $B$ is a local ring (Theorem I.4.2 on page 32 of \cite{Mi}), a contradiction to our assumption that $r>1$.
 \end{proof}
 
 As a consequence of the above theorem (Theorem \ref{ThmN2}), we now obtain Corollary \ref{CorN32}.
 
 \begin{Corollary} Let assumptions be as above and suppose that $R$ is a regular local ring. Suppose that $R\rightarrow R'$ is a nontrivial sequence of quadratic transforms along $\nu$. Then
 $\mbox{gr}_{\nu}(R')$ is not a finitely generated $\mbox{gr}_{\nu}(R)$-algebra.
 \end{Corollary}
 
 The conclusions of Theorem \ref{ThmN2} do not hold if we remove the assumption that $\nu^*$ is not discrete,  when $V_{\nu}/m_{\nu}$ is finite over  $R/m_R$. We give a  simple example. Let $k$ be an algebraically closed field of characteristic not equal to 2 and let $p(u)$ be a transcendental series in the power series ring $k[[u]]$ such that $p(0)=1$.  Then $f=v-up(u)$ is irreducible in the power series ring $k[[u,v]]$ and $k[[u,v]]/(f)$ is a discrete valuation ring with regular parameter $u$. Let $\nu$ be the natural valuation of this ring. 
 Let $R=k[u,v]_{(u,v)}$ and $S=k[x,y]_{(x,y)}$. Define a $k$-algebra homomorphism $R\rightarrow S$ by $u\mapsto x^2$ and $v\mapsto y^2$. The series
 $f(x^2,y^2)$ factors as $f=(y-x\sqrt{p(x^2)})(y+x\sqrt{p(x^2)})$ in $k[[x,y]]$. Let $f_1=y-x\sqrt{p(x^2)}$ and $f_2=y+x\sqrt{p(x^2)}$. The rings $k[[x,y]]/(f_i)$ are discrete valuation rings with regular parameter $x$. 
 Let $\nu_1$ and $\nu_2$ be the natural valuations of these ring. 
 
  Let $\nu$ be the valuation of the quotient field of $R$ which dominates $R$ defined by the natural inclusion $R\rightarrow k[[u,v]]/(f)$ and let $\nu_i$ for $i=1,2$  be the valuations of the quotient field of $S$ which dominate $S$ and are defined by the respective natural inclusions $S\rightarrow k[[x,y]]/(f_i)$ . Then $\nu_1$ and $\nu_2$ are distinct extensions of $\nu$ to the quotient field of $S$ which dominate $S$. However,
  we have that 
  $ {\rm gr}_{\nu}(R)=k[{\rm in}_{\nu}(u)]$ and ${\rm gr}_{\nu_i}(S)=k[{\rm in}_{\nu^*}(x)]$ with ${\rm in}_{\nu^*}(x)^2={\rm in}_{\nu}(u)$. Thus ${\rm gr}_{\nu_i}(S)$ is a finite  $ {\rm gr}_{\nu}(R)$-algebra.

 We now give an example where $\nu^*$ has rational rank 2 and  $\nu$  splits in $S$ but ${\rm gr}_{\nu^*}(S)$ is a finitely generated ${\rm gr}_{\nu}(R)$-algebra. Suppose that $k$ is an algebraically closed field of characteristic not equal to 2.
  Let $R=k[x,y]_{(x,y)}$ and $S=k[u,v]_{(u,v)}$.  The substitutions $u=x^2$ and $v=y^2$ make $S$ into a finite separable extension of $R$. 
  Define a valuation $\nu_1$ of the quotient field $K^*$ of $S$ by $\nu_1(x)=1$ and $\nu_1(y-x)=\pi+1$ and define a valuation $\nu_2$ of the quotient field $K^*$ by $\nu_2(x)=1$ and $\nu_2(y+x)=\pi+1$.
Since $u=x^2$ and $v-u=(y-x)(y+x)$, we have that $\nu_1(u)=\nu_2(u)=2$ and $\nu_1(v-u)=\nu_2(v-u)=\pi+2$. Let $\nu$ be the common restriction of $\nu_1$ and $\nu_2$ to the quotient field $K$ of $R$.
Then $\nu$ splits in $S$. However, ${\rm gr}_{\nu_1}(S)$ is a finitely generated ${\rm gr}_{\nu}(R)$-algebra since 
${\rm gr}_{\nu_1}(S)=k[{\rm in}_{\nu_1}(x),{\rm in }_{\nu_1}(y-x)]$ is a finitely generated $k$-algebra.  Note that  ${\rm gr}_{\nu}(R)=k[{\rm in}_{\nu}(u),{\rm in }_{\nu}(v-u)]$ with 
${\rm in}_{\nu_1}(x)^2={\rm in}_{\nu}(u)$ and ${\rm in}_{\nu}(v-u)=2{\rm in}_{\nu_1}(y-x){\rm in}_{\nu_1}(x)$.

\end{document}